\documentclass[11pt,reqno]{amsart}
\usepackage{amsfonts,amssymb,amsmath,amsthm}
\usepackage[margin=1in]{geometry}
\usepackage[hidelinks]{hyperref}
\usepackage{enumerate}
\usepackage{microtype}
\usepackage{mathtools}
\usepackage{cite}
\usepackage{color}
\usepackage{tikz-cd}
\usepackage{cleveref}
\usepackage{longtable}
\usepackage{float}

\definecolor{mylinkcolor}{rgb}{0.65,0.0,0.0}
\definecolor{myurlcolor}{rgb}{0.0,0.0,0.65}
\hypersetup{colorlinks=true,urlcolor=myurlcolor,citecolor=myurlcolor,linkcolor=mylinkcolor,linktoc=page,breaklinks=true}

\def\Z{\mathbb{Z}}
\def\C{\mathbb{C}}
\def\Q{\mathbb{Q}}

\def\P{\mathbb{P}}
\def\Qbar{\overline{\Q}}
\def\kbar{\overline{k}}
\def\Kbar{\overline{K}}
\def\F{\mathbb{F}}
\def\Gal{\operatorname{Gal}}
\def\GL{\operatorname{GL}}
\def\SL{\operatorname{SL}}

\def\det{\operatorname{det}}

\def\im{\operatorname{im}}

\def\Aut{\operatorname{Aut}}

\def\mc#1{\href{https://beta.lmfdb.org/ModularCurve/Q/#1/}{\texttt{#1}}}

\newcommand{\GalQ}{{\Gal}(\Qbar/\Q)}
\newcommand{\Zhat}{{\widehat{\Z}}}

\newcommand{\GalK}{{\Gal}(\Kbar/K)}
\newcommand{\calF}{{\mathcal F}}

\newcommand{\intersect}{\cap} 
\theoremstyle{plain}
\newtheorem{theorem}{Theorem}
\newtheorem{Corollary}[theorem]{Corollary}
\newtheorem{lemma}[theorem]{Lemma}
\newtheorem{proposition}[theorem]{Proposition}

\theoremstyle{definition}
\newtheorem{definition}[theorem]{Definition}
\newtheorem{example}[theorem]{Example}

\theoremstyle{remark}
\newtheorem{remark}[theorem]{Remark}

\crefname{section}{§}{§§}
\crefname{lemma}{Lemma}{Lemmas}
\crefname{equation}{equation}{equations}
\crefname{theorem}{Theorem}{Theorems}
\crefname{proposition}{Proposition}{Propositions}
\crefname{Corollary}{Corollary}{Corollaries}

\title{modular curves of prime power level with infinitely many quadratic points}

\author[Michael Cerchia and Rakvi]{Michael Cerchia and Rakvi}

\email{mjcerchia@gmail.com}
\email{rr657@cornell.edu}  
\date{\today}

\begin{document}

\begin{abstract} We completely determine the $1085$ open subgroups $H$ of $\GL_2(\Zhat)$ of prime power level that satisfy $-I \in H$ and $\det(H)=\Zhat^{\times}$ for which the corresponding modular curve $X_H$ has infinitely many quadratic points. When $g(X_H)\geq 2$ this is equivalent to determining all the hyperelliptic modular curves of prime power level and all the bielliptic modular curves of prime power level that have a degree two map to a positive rank elliptic curve. From the moduli perspective, this means that there are exactly 1085 subgroups $H$ of $\GL_2(\Zhat)$ of prime power level that contain $-I$ and have full determinant for which there are infinitely many elliptic curves $E/K$ over quadratic extensions such that $\rho_E(G_k)$ is conjugate to a subgroup of $H$.
\end{abstract}

\maketitle

\section{Introduction}

Let $X$ be a smooth projective curve defined over $\Q$. We say that $X$ is \textit{positive rank bielliptic} if there exists a degree two map over $\Q$ from $X$ to a positive rank elliptic curve, and we say that X is \textit{hyperelliptic} if it admits a degree two map over $\Q$ to $\P^1$. It follows from the work Hindry \cite{hindry1987points} (Remarque, p. 221) and Harris--Silverman \cite{silvermanharris} (Corollary $3$) that when $g_X\geq 2$ the set of quadratic points \[\Gamma_2(X,\Q):=\bigcup_{[F:\Q]\leq2}X(F)\] is infinite if and only if $X$ is either positive rank bielliptic or hyperelliptic. 

A subgroup $H\subseteq\GL_2(\Zhat)$ that contains $-I$ and has full determinant has an associated modular curve $X_H$ whose \textit{$\GL_2$-level} (or, more simply, \textit{level}) is defined to be the least positive integer $N$ such that $H$ contains the kernel of the reduction map $\GL_2(\Zhat)\to\GL_2(\Z/N\Z)$. In this paper, we completely classify the hyperelliptic and positive rank bielliptic modular curves of prime power level. This gives us a complete classification of the modular curves of prime power level that have infinitely many quadratic points. 

A key input is an a priori level bound in the prime-power case. We state it more precisely in the following theorem. 

\begin{theorem}[]\label{thm:boundB}
Let $H \subseteq \GL_2(\Zhat)$ be open with $-I \in H$ and $\det(H)=\Zhat^\times$, and assume that the level
of $H$ is a prime power. If $X_H$ has infinitely many quadratic points, then the level of $H$ is at most 131.
\end{theorem}

Theorem~\ref{thm:boundB} is proved in Section~\ref{sec:enumerate} by combining
Zywina's classification \cite{Zywina_lowgonality} (which bounds the relevant $\SL_2$-levels over $\C$)
with our Proposition~\ref{prop:boundingthelevel}, in which we convert $\SL_2$-level bounds into  $\GL_2$-level bounds.
Consequently, the classification reduces to studying the quadratic points of finitely many modular curves. We describe our strategy in \ref{strategy} below.

\subsection{Motivation and Previous Work}
Previous work on the classification of bielliptic modular curves has focused on the classical cases, such as Bars's work on $X_0(N)$ \cite{barsx0} and subsequent work on various quotients of $X_0(N)$ for varying levels (including \cite{biellquot}, \cite{biellquotsqfree}, \cite{biellatkins}). Similar results have also been obtained for $X_1(N)$ \cite{biellx1} and intermediate modular curves of the form $X_{\Delta}(N)$ where $\Delta$ is a subgroup of $(\Z/N\Z)^*$ \cite{biellint}. These results build on the much earlier work on the classification of hyperelliptic modular curves, most notably Ogg's \cite{ogg} classification of integers $N\geq 1$ for which $X_0(N)$ is hyperelliptic. 

Our motivation for studying quadratic points on more general modular curves comes from the following. Let $N\geq 1$ be an integer. To an elliptic curve $E$ over a number field $K$, one can associate a representation \[\rho_{E,N}:\Gal(\overline{K}/K)\to\GL_2(\Z/N\Z)\] induced by the action of the absolute Galois group $G_K:=\Gal(\overline{K}/K)$ on the $N$-torsion subgroup $E[N]$ of $E(\overline{K})$. After choosing compatible bases for the torsion subgroups $E[N]$ for $N\geq 1$, we can construct an adelic representation \[\rho_E:G_K\to\GL_2(\Zhat).\]

The following program on the classification of such representations -- often called ``Program B" -- was proposed by Mazur \cite{mazur}:
\begin{quote}
    ``Given a number field $K$ and a subgroup $H$ of $\GL_2(\Zhat) \simeq \prod_p \GL_2 (\Z_p)$, classify all elliptic curves $E/K$ whose associated Galois representation on torsion points maps $\Gal(\overline{K}/K)$ into $H \leq \GL_2 (\Zhat)$.''
\end{quote}

After the work of Rouse--Zureick-Brown \cite{rzb} and Rouse--Sutherland--Zureick-Brown\cite{rszb} that classified, respectively, the $2$-adic and $\ell$-adic (for $\ell\geq 3$) images of Galois for elliptic curves over $\Q$, the natural next case of Mazur's program is to determine the possible $\ell$-adic images of Galois as $K$ varies over all quadratic extensions. A modular curve $X_H$ loosely parametrizes elliptic curves whose associated Galois representation has image contained in $H$, and so the first step of this case of the project is completely determine which subgroups $H\subset\GL_2(\Zhat)$ (that contain $-I$ and have full determinant) with prime power level correspond to modular curves with infinitely many quadratic points. The analogous step over the rationals was completed by Sutherland and Zywina \cite{suthzyw}, who gave a complete list of the modular curves of prime power level with an infinitely many rational points. This step is necessary to construct a tree of \textit{maximally arithmetic subgroups}, as done in \cite{rzb} for the $2$-adic case over $\Q$ and in \cite{rszb} in the $\ell$-adic case, after which one is left with a finite number of modular curves with a finite number of points to analyze. 

\subsection{Strategy and Outline}\label{strategy}
We briefly summarize the computational steps.  After bounding the possible genus and level of a hyperelliptic modular curves of $\ell$-power level, we are left with a finite list of candidates. For most of these curves, we can use data from the LMFDB to quickly deduce whether or not they are in fact hyperelliptic. This leaves us with seven genus one curves, which are handled in \ref{remaininggenusone} by using techniques from Galois cohomology, and 60 curves of genus greater than or equal to two that we handle in Section \ref{sec:hyp}. Section \ref{remaininggenusone} also gives the first example (to our knowledge) of a genus one modular curve whose index is greater than its period. We then focus on the question of whether or not a candidate modular curve of $\ell$-power level is positive rank bielliptic. In principle, determining which of these curves are positive rank bielliptic can be done by finding all involutions $\iota$ of a given curve $X$, determining if the genus of $X/\langle\iota\rangle$ is one, and if it is, determining the rank of $X/\langle\iota\rangle$. In practice, this is often computationally difficult. Often (particularly when the genus is higher) it is slow to compute the automorphism group of these curves over $\Q$, and even when we can compute them, the genus one quotients that sometimes arise often have complicated models. On such models, it is difficult to find points (let alone determine ranks) and so we are led to construct simpler models of these quotient curves. We discuss this further in sections \ref{rankone} and \ref{nonrankone}, where we outline the techniques used to prove whether a given curve is positive rank bielliptic or not. These sections also contain examples of the application of these techniques. In the course of this classification, we prove the following.

\begin{theorem}
    Up to conjugacy, there are 1085 open subgroups $H$ of $\GL_2(\Zhat)$ of prime power level satisfying $-I\in H$ and $\det(H)=\Zhat^{\times}$ for which $X_H$ has infinitely many quadratic points. Of these 1085 groups,
    \begin{itemize}
        \item There are 265 groups of genus 0. Their labels are given in Table \ref{tab:genus 0}.
        \item There are 336 groups of genus 1. Their labels are given in Table \ref{tab:genus 1}.
        \item There are 306 groups of genus at least two that are hyperelliptic. Their labels are given in Table \ref{tab:hyp}.
        \item There are 178 groups of genus bigger than or equal to two that are non-hyperelliptic and positive rank bielliptic. Their labels are given in Tables \ref{tab:LMFDB deductions} and \ref{tab:posrankbi}. 
    \end{itemize}
    
\end{theorem}


\vspace{-4pt}

The following corollary follows from the moduli interpretation of modular curves. 
\begin{Corollary}
    There are exactly 1085 subgroups $H$ of $\GL_2(\Zhat)$ that contain $-I$, have full determinant and of prime power level for which there are infinitely many elliptic curves $E/K$ over quadratic extensions, with distinct $j$-invariants, such that $\rho_E(G_K)$ is conjugate to a subgroup of $H$.
\end{Corollary}

In a follow-up work, the authors are working to find all quadratic points on $X_H$ where $H$ is an open subgroup of $\GL_2(\Zhat)$ of prime power level for primes between 2 and 131 satisfying $-I\in H$ and $\det(H)=\Zhat^{\times}.$ With some group theory, it is then possible to describe all $\ell$-adic images for $2 \le \ell \le 131$ for $E/K.$  

\subsection{Comments on code} This paper has a computational component. Many of the candidate curves we consider have arguments that rely on \texttt{MAGMA} \cite{MR1484478}. Code verifying our claims is available at the Github repository 
\begin{center}
\url{https://github.com/mjcerchia/quadratic-points-on-modular-curves}.
\end{center}
We used V2.29-5 of \texttt{MAGMA} for our computations.
We note that our results rely on beta version of the LMFDB. Specifically, we rely on completeness of the database for modular curves of prime power levels up to 335. We use the models if they are given on LMFDB. We also use the data on Jacobian decomposition.

\section{Acknowledgments}
We thank Jeremy Rouse, David Zureick-Brown, and David Zywina for their helpful
suggestions and several valuable conversations. We thank Edgar Costa for helping with computing an involution for \mc{27.108.7.a.1}. We also thank Maarten Derickx for suggesting a proof strategy for Proposition \ref{prop:index4} and for giving us comments on an earlier draft of this paper. Finally, we thank an anonymous referee for helpful comments, which led to an improvement of the exposition of this paper. All computations were done either in \texttt{MAGMA} or in \texttt{SAGE} \cite{sagemath}.

\section{Background and preliminary results}\label{background}

We present some definitions and preliminary results here that will be needed in later sections. 
\subsection{Subgroups of $\GL_2(\Zhat)$ and Congruence Subgroups}

Let $H$ be an open subgroup of $\GL_2(\Zhat).$ 
We say that $N$ is the \textit{$\GL_2$-level} (or, more simply, \textit{level}) of $H$ if $N$ is the smallest integer such that $H$ contains the kernel of the natural projection map $\pi_N:\GL_2(\hat{\Z})\to\GL_2(\Z/N\Z)$. If $H$ is an open subgroup of $\GL_2(\Zhat)$ of $\GL_2$ level $N$ we will denote $\pi_N(H)$ by $H_N.$

For any positive integer $M$, let $\phi_M$ be the natural projection map $\phi_M:\SL_2(\Z)\to\SL_2(\Z/M\Z)$. Let $\Gamma$ be a congruence subgroup of $\SL_{2}(\Z).$ We say that $M$ is the \textit{$\SL_2$-level} of $\Gamma$ if it is the smallest positive integer for which $\Gamma$ contains the kernel of $\phi_M.$ We will write $\Gamma_M$ to denote $\phi_M(\Gamma).$

For an open subgroup $H$ of $\GL_2(\Zhat)$ of $\GL_2$-level $N$ there is a corresponding congruence subgroup $\Gamma_H$ which can be defined as $\Gamma_H \coloneqq \phi_N^{-1}(H_N \intersect \SL_2(\Z/N\Z)).$ 

\begin{lemma}
 Let $H$ be an open subgroup of $\GL_2(\Zhat)$ of $\GL_2$-level $N$. Then the $\SL_2$ level of $\Gamma_H$ divides $N.$
\end{lemma}

\begin{proof}
    Let $M$ be the $\SL_2$ level of $\Gamma_H$. By definition of $\Gamma_H$, it contains the kernel of $\phi_N$, and so $M$ divides $N.$ 
\end{proof}

\begin{remark}
    The $\GL_2$ level of $H$ can be arbitrarily large multiple of the $\SL_2$ level of $\Gamma_H$. Consider the following example. Let $d$ be a square free positive integer. Let $\epsilon \colon \GL_2(\Zhat) \to \{\pm 1\}$ be the composition of the reduction map $\GL_2(\Zhat) \to \GL_2(\Z/2\Z)$ and the signature map $\GL_2(\Z/2\Z) \to S_3 \to \{ \pm 1\}.$ Let $\chi_d \colon \GalQ \to \Gal(\Q(\sqrt{d})/\Q) \to \{\pm 1\}$ be the cyclotomic character. Define $H_d \coloneqq \{A \in \GL_2(\Zhat)~|~\epsilon(A)=\chi_d(det(A)\}.$ Then the $\SL_2$ level of $\Gamma_H$ is equal to two for all $d$ but the $\GL_2$ level of $H_d$ is a multiple of $d$.
\end{remark}

\subsection{Modular functions}

Fix a positive integer $M.$ 

The group $\SL_2(\Z)$ acts on the extended upper half plane by linear fractional transformations. Let $\Gamma(M)$ be the kernel of $\phi_M.$ After adding cusps to the quotient of upper half plane by $\Gamma(M)$, we get a smooth compact Riemann surface which we denote by $X_M.$ 

Let $\calF_M$ denote the field of meromorphic functions of the Riemann surface $X_M$ whose $q$-expansion has coefficients in the $M$-th cyclotomic field which we will denote by $K_{M}$. For $M=1$, we have $\calF_1=\Q(j)$ where $j$ is the modular $j$-invariant. The first few terms of $q$-expansion of $j$ are $q^{-1}+744+196884q+21493760q^2+\dotsb.$
If $M'$ is a divisor of $M$, then $\calF_{M'} \subseteq \calF_M.$ In particular, we have that $\calF_1 \subseteq \calF_M.$

There is a natural action of $\SL_2(\Z/M\Z)$ on $F_M.$ With respect to this action, we have \[\mathcal{F}_M^{\SL_{2}(\Z/M \Z)/\{\pm I\}}=K_M(j).\] For more details on this, please refer to Chapter $6$ in \cite{MR1291394}.

\subsection{Modular Curves}\label{sec:modularcurves}

The aim of this section is to provide an overview of modular curves and its properties that are relevant to us. For a further background on this topic the reader is encouraged to read chapter $4$, section 3 of \cite{MR0337993}.

Let $H$ be an open subgroup of $\GL_2(\Zhat)$ that contains $-I$ of $\GL_2$ level $N$. There is a natural isomorphism between $\Gal(K_N/\Q)$ and $(\Z/N\Z)^{\times}.$ 

Let $K_H$ be the unique subfield of $K_N$ such that $\Gal(K_N/K_H)$ is isomorphic to $\det(H_N)$. For example, if $\det(H)=\Zhat^{\times}$, then $K_H=\Q.$ Associated to $H$ there is a nice (smooth, projective, geometrically integral) curve $X_H$ which is defined over $K_H$ and which parametrizes elliptic curves with $H$-level structure. Roughly, $H$ parametrizes elliptic curves whose adelic Galois image lands in $H$. We call $X_H$ the modular curve associated to $H$.

For a congruence subgroup $\Gamma$ of $\SL_2$ level $M$ that contains $-I$, $X_{\Gamma}$ is the nice curve over $K_M$ whose function field is $\calF_M^{\Gamma_M}.$ The map $\pi_\Gamma \colon X_{\Gamma} \to \P^1_{K_M}$ is the morphism corresponding to the inclusion $K_M(j) \subseteq \calF_M^{\Gamma_M}.$ Define $\Aut(X_{\Gamma},\pi_{\Gamma}) \coloneqq \{f \colon X_{\Gamma} \to X_{\Gamma} | \pi_{\Gamma} \circ f = \pi_{\Gamma}\}$. It is straightforward to check that $\Aut(X_{\Gamma},\pi_{\Gamma})$ is a group defined over $K_M$.

\begin{remark}
    For convenience, we sometimes ascribe geometric terms to a subgroup $H$ rather than the corresponding modular curve $X_H$. That is, we may call a group $H$ ``hyperelliptic" or ``positive rank bielliptic" or refer to its genus.
\end{remark}

\begin{theorem}
  Let $H$ be an open subgroup of $\GL_2(\Zhat)$ that contains $-I$. Let $N$ be its $\GL_2$ level. Let $X_H$ be the modular curve associated to $H$. The following properties hold:

    \begin{itemize}
        \item \label{thm:modcurves1}If $H \subseteq H'$, then there is a morphism $X_H \to X_{H'}.$ 
        \item \label{thm:modcurves2}There is a natural map $\pi_H \colon X_H \to \P^1$ associated to $X_H.$
        \item \label{thm:modcurves3}For a non CM elliptic curve $E$ defined over a number field $K$ we have $\rho_E(\GalK)$ is conjugate to a subgroup of $H$ if and only if $K$ contains $K_H$ and the $j$-invariant of $E$ lies in $\pi_H(X_H(K)).$
        \item \label{thm:modcurves4} There exists an isomorphism $f \colon X_{\Gamma_H} \to X_{H}$ over $K_N$ such that $\pi_H \circ f = \pi_{\Gamma_H}.$ 
        
    \end{itemize}  
\end{theorem}

\begin{proof}
    Please see Section 2.2 of \cite{MR4833876} for first three properties and Theorem $6.6$, Chapter $6$ in \cite{MR1291394} for a proof of last property.
\end{proof}

We use the following proposition to bound $\GL_2$-level of modular curves. Our proof follows the strategy of Lemma 8.1 \cite{MR4730250}.
\begin{proposition}\label{prop:boundingthelevel}

Let $H$ be an open subgroup of $\GL_2(\Zhat)$ that contains $-I$ and has full determinant. Let $X_H$ be the associated modular curve. Let $N$ be the $\GL_2$ level of $H$. Let $M$ be the $\SL_2$ level of $\Gamma_H.$ We also assume that $N$ and the $M$ have same prime factors. Then, there exists an upper bound $b$ for $N$ that can be explicitly computed from the order of $\Aut(X_{\Gamma_H},\pi_{\Gamma_H})$ and $M.$ More explicitly, $b$ is the least common multiple of orders of elements $a \in \Aut(X_{\Gamma_H},\pi_{\Gamma_H})$ that divide some power of $M$. 
\end{proposition}

\begin{proof}
    Let $M$ be the $\SL_2$ level of $\Gamma_H.$ Then, $N$ is the smallest positive multiple of $M$ for which we can find an isomorphism $f \colon (X_{\Gamma_H})_{K_N} \to (X_H)_{K_N}$ that satisfies $\pi_H \circ f = \pi_{\Gamma_H}.$ For $\sigma \in \Gal(K_{N^{\infty}}/K_M)$, $f^{-1}\sigma f$ is a map from $(X_{\Gamma})_{K_{N^{\infty}}} \to (X_{\Gamma})_{K_{N^{\infty}}}$ that satisfies $\pi_{\Gamma_H} \circ (f^{-1}\sigma f)=\pi_{\Gamma_H}.$ Therefore, we can define a homomorphism $\phi \colon \Gal(K_{N^{\infty}}/K_M) \to \Aut(X_{\Gamma_H},\pi_{\Gamma_H})$ that sends $\sigma$ to $f^{-1}\sigma f.$ 

    Let $b$ be the least common multiple of orders of elements $a \in \Aut(X_{\Gamma_H},\pi_{\Gamma_H})$ that divide some power of $M$. Define $G:=\Gal(K_{N^{\infty}}/K_M)$ if $M$ is not congruent to $2$ modulo $4$ and $G:=\Gal(K_{N^{\infty}}/K_{2M})$ if $M$ is congruent to $2$ modulo $4.$ The Galois group $G$ is cyclic and its cardinality has the same prime factors as $M$. Let $g$ be a generator of $G.$ Let $a:=\phi(g).$ Then by our choice of $b$ we get that $a^b=1.$ Since $\phi$ is a homomorphism, $\phi(g^b)=1.$ It follows that $\phi$ factors through $\Gal(K_{Mb}/K_M)$ if $M$ is not congruent to $2$ modulo $4$ and $\phi$ factors through $\Gal(K_{2Mb}/K_M)$ if $M$ is congruent to $2$ modulo $4.$ This gives us an upper bound on $N.$ 
\end{proof}

\subsection{Jacobian Decomposition}

Throughout our computations, we make use of the fact that the Jacobian of modular curve has a decomposition corresponding to certain modular forms. Let $H$ be an open subgroup of $\GL_2(\Zhat)$, and let the Jacobian of the corresponding modular curve $X_H$ be given by $J_H:=\operatorname{Jac}(X_H)$. This is an abelian variety that has good reduction at all primes $p$ that do not divide the level $N$ of $H$. As explained in Theorem A.7 of \cite{rszb}, the simple factors of $J_H$ are $\Q$-isogenous to the modular abelian variety associated to the Galois orbit of a weight $2$ eigenform for $\Gamma_1(N)\cap\Gamma_0(N^2).$ For most of the modular curves we are interested in, the Jacobian decompositions are given on the LMFDB \cite{lmfdb}. These decompositions are useful because they tell us which elliptic curves a modular curve could possibly map to. In particular, we reduce the size of our candidate list of positive rank bielliptic curves by ruling out any non-hyperelliptic modular curves that do not have a rank one elliptic curve in their Jacobian decompositions. In some cases, we compare $L$-functions of genus one quotients by an involution of a modular curve to any rank one elliptic curve factors of the Jacobian. This sometimes allows us to conclude that a given modular curve cannot be positive rank bielliptic. We repeatedly make use of the following fact. 

\begin{lemma} \label{lemma:jacfactor}
    Let $C$ be a nice curve over $k$ and suppose that $f: C\to E$ is a morphism of degree $d$ to an elliptic curve $E$. Then $E$ appears (up to isogeny) in the isogeny decomposition of the Jacobian of $C$.
\end{lemma}
\begin{proof}
Suppose $f:C\to E$ is a such a map. Identify $E\cong J(E)\cong\operatorname{Pic}^0(E)$ via $x\mapsto [x-O_E].$ The finite map $f$ induces homomorphisms of abelian varieties $f^*:J(E)\to J(C)$ and $f_*:J(C)\to J(E)$ given by the pullback and pushforward of degree $0$ divisors, respectively. These homomorphisms satisfy $f_*\circ f^*=[d]_{J(E)}$, where the right hand side denotes the isogeny of multiplication by $d$. We have that $\ker(f^*)\subseteq\ker[d]_{J(E)}=E[d]$, which is finite. This gives us an isogeny of elliptic curves $E/\ker(f^*)\to\im(f^*)$, and composing with the quotient $E\to E/\ker(f^*)$ shows that the abelian subvariety $\im(f^*)$ of $J(C)$ is isogenous to $E$. By Poincaré complete irreducibility (Cor 3.20 of \cite{conrad}), there exists an abelian variety $B$ with $J(C)\sim\im(f^*)\times B\sim E\times B$.
\end{proof}

We repeatedly use the following fact, particularly when the codomain is an elliptic curve. 
\begin{lemma}\label{lemma:involutionrmk}
     Suppose $X$ and $C$ are nice curves over field $k$ of characteristic $0$ and $\varphi:X\to C$ is a degree two map over $k.$ Then $X/\langle\iota\rangle\cong C$ for some involution $\iota$ of $X.$
    \begin{proof}
        Note that $\varphi$ induces a degree two Galois extension of function fields $k(X)/k(C)$, and by general Galois theory for curves, the nontrivial element $\iota^*$ of $\Gal(k(X)/k(C))$ comes from a unique involution $\iota:X\to X$ such that on functions, $\iota^*(f)=f\circ\iota$. It then follows that $k(X/\langle\iota\rangle)=k(X)^{\langle\iota^*\rangle}=k(C)$, where the middle term is the fixed field of $\langle\iota\rangle$. Since a smooth projective curve is determined uniquely (up to unique isomorphism) by its function field, we have $X/\langle\iota\rangle\cong C$.  
    \end{proof}
\end{lemma}

We use the following lemma to deal with some special cases.

\begin{lemma}\label{lemma:nfactors}
    Let $C$ be a nice curve over $k$ of genus $5$. Suppose an elliptic curve $E$ appears $n$ number of times in the Jacobian decomposition of $C$, where $n$ is either $1$ or $2.$ Further suppose that there exist $n$ degree two maps $d_1: C \to C_1$, \ldots, $d_n: C \to C_n$ where $C_i$ is a pointless genus 1 curve satisfying $J(C_i)=E$ for $i \in \{1,\ldots,n\}.$ Then, there does not exist a degree $2$ map from $C$ to $E.$
\end{lemma}

\begin{proof}
    We will assume that $n=1$. Suppose that there exists a degree $2$ map $f: C \to E.$ By lemma \ref{lemma:involutionrmk} we know that there exist two involutions $\sigma_1$ and $\sigma_2$ that correspond to $C_1$ and $E$, respectively. Call the group generated by $\sigma_1$ and $\sigma_2$ G. From lemma 2.3 \cite{MR2853619} we know that G is a Klein four group and $C/G$ has genus $0$. In Theorem 2 \cite{MR2403651} if we take $H_1=\langle\sigma_1\rangle$, $H_2=\langle\sigma_2\rangle$ and $H_3=\langle\sigma_1\sigma_2\rangle$ we get a contradiction from number of times $E$ appears in the Jacobian decomposition of $C$. 

    We now assume that $n=2$. Suppose that there exists a degree $2$ map $f: C \to E.$ Again by lemma \ref{lemma:involutionrmk} we know that there exist three involutions $\sigma_1$, $\sigma_2$ and $\sigma_3$ that correspond to $C_1$, $C_2$ and $E$, respectively. Call the group generated by $\sigma_i$ G, where $i \in \{1,2,3\}$. From lemma 2.3 \cite{MR2853619} we know that G is an elementary abelian group of order 8 and $C/G$ has genus $0$. In Theorem 2 \cite{MR2403651} if we take $H_1=\langle\sigma_1 \rangle$, $H_2=\langle\sigma_2\rangle$ and $H_3=\langle\sigma_1\sigma_2\rangle$ we can conclude that $C/\langle\sigma_1\sigma_2\rangle$ has no $E$ in its Jacobian decomposition. Similarly, $C/\langle\sigma_i\sigma_j\rangle$ has no $E$ in its Jacobian decomposition for $i \ne j$ and $i,j \in \{1,2,3\}.$ Now in Theorem 2 \cite{MR2403651} if we take $H_1=\langle\sigma_1\rangle$, $H_2=\langle\sigma_2\rangle$, $H_3=\langle\sigma_3\rangle$, $H_4=\langle\sigma_1\sigma_2\rangle$, $H_5=\langle\sigma_2\sigma_3\rangle$, $H_6=\langle\sigma_1\sigma_3\rangle$ and $H_7=\langle\sigma_1\sigma_2\sigma_3\rangle$, then we get a contradiction from number of times $E$ appears in the Jacobian decomposition on both sides. 
\end{proof}

\subsection{Gonality}
We collect some facts about curve gonality that will be useful to us in the next section when we reduce to a finite list of candidate hyperelliptic and positive rank bielliptic modular curves. 

\begin{definition} Let $C$ be a nice curve defined over a field $k$. The \textit{$k$-gonality} of $C$ is defined as the minimum degree of a morphism $\phi \colon C \to \P^1_{\Q}$, where $\phi$ is defined over $k.$
\end{definition}

\begin{lemma}\label{lemma:gonalitybasechange}
    Let $C$ be a nice curve defined over a field $k$. If $k$ is a subfield of $K$, then $K$-gonality of $C$ is less than or equal to $k$-gonality of $C.$
\end{lemma}

\begin{proof}
    This immediately follows from our definition of gonality and the observation that any morphism defined over $k$ is also defined over $K.$
\end{proof}

\begin{definition}
    Let $k$ be a field of characteristic $0$. Let $C$ be a nice curve over $k$. We say that $C$ is \textit{hyperelliptic} if there exists a degree $2$ map over $k$ from $C$ to $\P^1_k.$
\end{definition}

\begin{lemma}\label{lemma:hyperellipticgonality2}
    Let $k$ be a field of characteristic $0$. Let $C$ be a nice curve over $k$ of genus greater than or equal to $1.$ Then, $C$ is hyperelliptic if and only if its $k$-gonality is $2.$
\end{lemma}

\begin{proof}
    Assume that $C$ is hyperelliptic. Then there exists a degree $2$ map over $k$ from $C$ to $\P^1_k.$ Further, there cannot be a degree $1$ map from $C$ to $\P^1_k$ because that would imply that genus of $C$ is $0$. Therefore, its $k$-gonality is $2.$

    If the $k$-gonality of $C$ is $2$, then there exists a degree $2$ map over $k$ from $C$ to $\P^1_k.$ By definition, $C$ is hyperelliptic.
\end{proof}

\begin{definition}
    Let $k$ be a field of characteristic $0$. Let $C$ be a nice curve over $k$ of genus greater than or equal to $2.$ We say that $C$ is \textit{positive rank bielliptic} if there exists a degree two map over $k$ from $C$ to a positive rank elliptic curve $E$ over $k$.
\end{definition}

\begin{definition}
    Let $k$ be a field of characteristic $0$. Let $C$ be a nice curve over $k$. We say that $P \in C(\kbar)$ is a \textit{quadratic point} of $C$ if the degree of $k(P)$ over $k$ is at most two.
\end{definition}

\begin{theorem}\label{them:infquadpoints}
    Let $k$ be a field of characteristic $0$. Let $C$ be a nice curve over $k$ of genus greater than or equal to $2$. There exist infinitely many quadratic points of $C$ if and only if $C$ is hyperelliptic or positive rank bielliptic.
\end{theorem}

\begin{proof}
    If $C$ is hyperelliptic or positive rank bielliptic, then it is easy to observe that there are infinitely many quadratic points of $C$. For converse, please see corollary 3 \cite{silvermanharris}. 
\end{proof}

We now state the Castelnuovo–Severi inequality [Theorem 3.11.3 \cite{MR2464941}] and following it we give some corollaries that we will use later to decide if a curve is hyperelliptic/positive rank bielliptic or not.

\begin{theorem}\label{thm:CSineq}
    Let $k$ be a field of characteristic $0.$ Let $\phi_1 \colon C \to C_1$ and $\phi_2 \colon C \to C_2$ be nonconstant morphisms, respectively, where $C_1$ and $C_2$ are nice curves over $k.$ Assume there is no morphism $\phi \colon C \to C'$ of degree greater than one through which both $\phi_1$ and $\phi_2$ factor. Then 
    \[g \le d_1g_1 +d_2g_2 +(d_1 -1)(d_2 -1),\] where $g_i$ is the genus of $C_i$ and $d_i$ is the degree of $\phi_i$ for $i \in \{1,2\}.$
\end{theorem}

\begin{Corollary}\label{cor:unique genus 0}
    Let $C$ be a nice curve over a field $k$ of char $0$ of genus $g$ bigger than or equal to two. If there exists a degree 2 map $\phi \colon C \to C'$ where $C'$ is a pointless conic over $k$, then there cannot exist a degree $2$ map $\psi \colon C \to \P^1_k.$ 
\end{Corollary}

\begin{proof}
    Suppose there exists a degree $2$ map $\psi \colon C \to \P^1_k.$ Then there is no morphism $\eta \colon C \to C''$ of degree bigger than one through which both $\phi$ and $\psi$ factor because that would imply $C'$ is isomorphic to $\P^1_k$ which would be a contradiction. From Theorem \ref{thm:CSineq} we know that $g \le 1$ which is a contradiction because of our assumption on genus of $C$.
\end{proof}




\begin{lemma}
   Let $X$ be a nice curve defined over $k.$ Let $p$ be a prime of good reduction of $X$. Let $X_p$ denote the reduction of $X$ at $p.$ If $X_p$ has no genus $1$ quotient by involution, then $X$ is not bielliptic. 
\end{lemma}

\begin{proof}
    If $X$ is bielliptic, then there exists a map $\phi \colon X \to E$ where $E$ is a genus $1$ curve. As discussed in Lemma \ref{lemma:involutionrmk} we know that $\phi$ is the quotient map by some involution of $X$. Reducing this at $p$, we observe that $E_p$ must be a quotient of $X_p$ by involution. The lemma now follows.
\end{proof}

We are interested in determining all the subgroups $H$ of $\GL_2(\Zhat)$ of prime power level that contain $-I$ and satisfy $\det(H)=\Zhat^{\times}$ such that there are infinitely many elliptic curves $E$ over quadratic fields that have the property $\rho_E(\GalK) \subseteq H.$ From Theorem \ref{them:infquadpoints} we know that this is equivalent to determining if $X_H$ is is hyperelliptic or positive rank bielliptic when genus of $X_H$ is greater than or equal to two. We handle genus 0 and genus 1 cases separately.

\section{Enumerating subgroups to check for hyperelliptic and bielliptic modular curves}\label{sec:enumerate}

Let $H$ be an open subgroup of $\GL_2(\Zhat)$ of prime power level that contains $-I$ and has full determinant. Let $X_H$ be the associated modular curve. In this section we will discuss enumeration of all such subgroups for which $X_H$ is possibly either hyperelliptic or positive rank bielliptic. To get a list of all groups we rely on the completeness of LMFDB database of modular curves, which we explain below.

If $H$ is an open subgroup of $\GL_2(\Zhat)$ that contains $-I$ and satisfies $\det(H)=\Zhat^{\times}$ such that the $\GL_2$ level of $H$ has prime power level at most 335 and $X_H$ has genus at most 24, then $H$ is in LMFDB database of modular curves.

\subsection{Genus 0} If $X_H$ has genus $0$, then it has infinitely many quadratic points. So all we need to do is enumerate all such subgroups. The Cummins-Pauli labels of corresponding congruence subgroups are

Genus0CP:=\{\texttt{2A0}, \texttt{2B0}, \texttt{2C0}, \texttt{3A0}, \texttt{3B0}, \texttt{3C0}, \texttt{3D0}, \texttt{4A0}, \texttt{4B0}, \texttt{4C0}, \texttt{4D0}, \texttt{4E0}, \texttt{4F0}, \texttt{4G0}, \texttt{5A0}, \texttt{5B0}, \texttt{5C0}, \texttt{5D0}, \texttt{5E0}, \texttt{5F0}, \texttt{5G0}, \texttt{5H0}, \texttt{7A0}, \texttt{7B0}, \texttt{7C0}, \texttt{7D0}, \texttt{7E0}, \texttt{7F0}, \texttt{7G0}, \texttt{8A0}, \texttt{8B0}, \texttt{8C0}, 
    \texttt{8D0}, \texttt{8E0}, \texttt{8F0}, \texttt{8G0}, \texttt{8H0}, \texttt{8I0}, \texttt{8J0}, \texttt{8K0}, \texttt{8L0}, \texttt{8M0}, \texttt{8N0}, \texttt{8O0}, \texttt{8P0}, \texttt{9A0}, \texttt{9B0}, \texttt{9C0}, \texttt{9D0}, \texttt{9E0}, \texttt{9F0}, \texttt{9G0}, \texttt{9H0}, \texttt{9I0}, \texttt{9J0}, \texttt{11A0},  \texttt{13A0}, 
    \texttt{13B0}, \texttt{13C0}, \texttt{16A0}, \texttt{16B0}, \texttt{16C0}, \texttt{16D0}, \texttt{16E0}, 
    \texttt{16F0}, \texttt{16G0}, \texttt{16H0},  \texttt{25A0}, 
    \texttt{25B0}, \texttt{27A0}, \texttt{32A0}\}.

For each of the congruence groups whose label is in the set Genus0CP we use Proposition \ref{prop:boundingthelevel} to get an upper bound on the level of $H$. If the level of congruence subgroup is $p^n$ for some natural number $n$, then an upper bound on the $\GL_2$ level of $H$ is summarized below.

\begin{itemize}
    \item For $p=2$, an upper bound is $256.$
    \item For $p=3$, an upper bound is $81.$
    \item For $p=5$, an upper bound is $125.$
    \item For $p \in \{7,11,13\}$, the upper bound is $p.$
\end{itemize}

These computations can be verified in the file \href{https://github.com/mjcerchia/two-adic-galois-images-quadratic/blob/main/Genus%200}{\texttt{Genus 0}}.

Since the upper bound is less than 335 all such modular curves are in LMFDB database. It is easy to observe that the maximum $\GL_2$ level of associated group is 32. There are 265 such curves and their LMFDB labels are given in Table \ref{tab:genus 0}.

\subsection{Genus 1} All our computational claims for this subsection can be verified in \href{https://github.com/mjcerchia/two-adic-galois-images-quadratic/blob/main/Genus1}{\texttt{Genus1}}.

The Cummins Pauli labels \cite{MR2016709} of congruence subgroups of prime power level are 

Genus1CP:=\{ \texttt{7A1}, \texttt{7B1}, \texttt{7C1}, \texttt{8A1}, \texttt{8B1}, \texttt{8C1}, \texttt{8D1}, \texttt{8E1}, \texttt{8F1}, 
    \texttt{8G1}, \texttt{8H1}, \texttt{8I1}, \texttt{8J1}, \texttt{8K1}, \texttt{9A1}, \texttt{9B1}, \texttt{9C1}, \texttt{9D1}, \texttt{9E1}, \texttt{9F1}, \texttt{9G1}, \texttt{9H1},  \texttt{11A1}, \texttt{11B1}, \texttt{11C1}, \texttt{11D1}, 
    \texttt{16A1}, \texttt{16B1}, \texttt{16C1}, \texttt{16D1}, \texttt{16E1}, \texttt{16F1}, \texttt{16G1}, \texttt{16H1}, \texttt{16I1}, \texttt{16J1}, \texttt{16K1}, \texttt{16L1}, \texttt{16M1}, 
    \texttt{17A1}, \texttt{17B1}, \texttt{17C1}, \texttt{19A1}, \texttt{19B1}, \texttt{27A1}, \texttt{27B1}, \texttt{27C1}, \texttt{32A1}, \texttt{32B1}, \texttt{32C1}, \texttt{32D1}, \texttt{32E1}, \texttt{49A1}\}.

For each of the congruence groups whose label is in the set Genus1CP we use Proposition \ref{prop:boundingthelevel} to get an upper bound on the level of $H$. If the level of congruence subgroup is $p^n$ for some natural number $n$, then an upper bound on the $\GL_2$ level of $H$ is summarized below.

\begin{itemize}
    \item For $p=2$, an upper bound is $128.$
    \item For $p=3$, an upper bound is $81.$
    \item For $p=7$, an upper bound is $49.$
    \item For $p \in \{11,17,19\}$, the upper bound is $p.$
\end{itemize}

Since the upper bound is less than 335 all such modular curves are in LMFDB database. There are 342 such curves and their LMFDB labels are given in \href{https://github.com/mjcerchia/two-adic-galois-images-quadratic/blob/main/Genus1}{\texttt{Genus1}}. In genus 1 case, the maximum $\GL_2$ level of associated group is 49. Out of these 342 curves there are 335 curves that are elliptic curves or admit a degree 2 map to $\P^1_{\Q}.$ This information is also available on LMFDB. The labels of remaining seven curves are \{\mc{8.48.1.bi.1}, \mc{9.81.1.a.1}, \mc{16.96.1.t.1}, \mc{16.48.1.l.1}, \mc{16.64.1.a.1}, \mc{16.64.1.b.1}, \mc{16.96.1.k.1}\}. We discuss them below.

\subsubsection{The remaining seven genus one cases}\label{remaininggenusone}

Let $K$ be a perfect field, let $E/K$ be an elliptic curve, and let $C/K$ be a homogeneous space for $E/K$. We recall that the \textit{period} of $C/K$ is the exact order of $\{C/K\}$ in the Weil-Chatelet group $\operatorname{WC}(E/K)$, and the \textit{index} of $C/K$ is the smallest degree of an extension $L/K$ such that $C(L)\neq\emptyset$.

To handle the remaining cases, we use the following well known fact. 

\begin{proposition}\label{prop:genus1cohom} Let $C$ be a genus one curve over a perfect field $K$. The the period of $C/K$ divides the index. 

\begin{proof}
    See for instance chapter $X$ of \cite{silverman2009arithmetic}.
\end{proof}
    
\end{proposition}

For computational purposes described below, we also rely on the following result. 

\begin{proposition}
    Let $E/K$ be an elliptic curve. There is a natural bijection \[\operatorname{WC}(E/K)\to H^1(G_K,E).\]
\end{proposition}
\begin{proof}
    The construction of this map and a proof can be found in Theorem X.3.6 of \cite{silverman2009arithmetic}
\end{proof}

For curves with labels in \{\mc{16.96.1.t.1},  \mc{16.64.1.a.1}, \mc{16.64.1.b.1}, \mc{16.96.1.k.1}\} we verify that $[C]$ does not represent an element of $H^1(G_K,E[2])$ by checking that $2[C]$ has no rational points due to local obstructions. So from Proposition \ref{prop:genus1cohom} there are no quadratic points on any of these.

For \mc{9.81.1.a.1} we can see from LMFDB that there is a degree 3 map from it to $\P^1_{\Q}$ so from Proposition \ref{prop:genus1cohom} $[C]$ represents an element of $H^1(G_K,E[3])$ so if it also represents an element of $H^1(G_K,E[2])$, then it should have rational points. We verify that it has no $\Q_3$ points, hence it does not represent an element of $H^1(G_K,E[2])$ so from Proposition \ref{prop:genus1cohom} there are no quadratic points on it.

For \mc{8.48.1.bi.1} and \mc{16.48.1.l.1} we observe in both cases that the corresponding curve class $[C]$ represents an element of $H^1(G_K,E[2])$ but from this we cannot conclude that it has quadratic points. Below, we show that \mc{8.48.1.bi.1} has index four, and hence no quadratic points. As far as we know, this is the only known example of a genus $1$ modular curve where the index is larger than the period. Our argument expands on one used by Cassels \cite{cassels} who found the first example of a genus $1$ curve with an index larger than its period.  

The labels of genus 1 modular curves that have infinitely many quadratic points are given in Table \ref{tab:genus 1}.

\subsubsection{8.48.1.bi.1} \label{periodindex}
\begin{proposition}\label{prop:index4}
The modular curve with LMFDB label 8.48.1.bi.1 has no quadratic points. In particular, it has index greater than 2.
\begin{proof}
    Call this curve $C$. We show that $C$ has no quadratic points by showing that it admits no degree two $\Q$-rational effective divisors. This is sufficient, because if there existed a point $P\in C(K)$ for some degree two extension $K/\Q$, then by letting $\sigma$ be the nontrivial element of $\Gal(K/\Q)$ we would have that $D:=P+\sigma(P)$ is an effective degree $2$ divisor on $C_\Q$ that is fixed by $\Gal(K/\Q)$, hence defined over $\Q$.
    
    From the LMFDB, we observe that $C$ is given by the intersection of two quadrics in $\P^3$: \begin{align*}
    Q_1:2x^2 + 3xy + yz - z^2 + w^2&=0 \\
    Q_2:4x^2 - 2xy + y^2 - 2yz + 2z^2&=0.
\end{align*} Its Jacobian is an elliptic curve given by $E:y^2 = x^3 + 36x^2 - 272x + 448$, which has exactly four rational points, and so there can be at most four degree two rational divisor classes on C. Since the period of $C$ is $2$, $\operatorname{Pic}^2(C)$ is a trivial $E$-torsor, and so we have that $\operatorname{Pic}^2(C)\cong E(\Q)\cong \langle P\rangle=\{O,P,-P,2P\}\cong\Z/4Z$ for an appropriate choice of $P\in E(\Q)$. There are thus four $\Q$-rational divisor classes of degree $2$, which we will call $[D_O],[D_P],[D_{2P}]$, and $[D_{-P}]$. It thus suffices to show that none of these divisor classes contains a rational divisor.



 To that end, we first determine the singular members of the pencil of quadrics containing $C$. The two quadrics defining $C$ have the respective symmetric matrices \[A_1=\begin{pmatrix}
    2&3/2&0&0\\
    3/2 & 0&1/2&0\\
    0&1/2&-1&0\\
    0&0&0&1
\end{pmatrix}\text{ and } A_2=\begin{pmatrix}
    4&-1&0&0\\
    -1&1&-1&0\\
    0&-1&2&0\\
    0&0&0&0
\end{pmatrix},\] which satisfy $Q_i(X)=X^TA_iX$ for $X=(x,y,z,w)$. Now consider the pencil given by $M(u,v)=uA_1+vA_2$ where $[u:v]\in\P^1$. A member $Q_{u:v}$ is singular if and only if $\det M(u,v)=0$. A direct calculation yields \[\det M(u,v)=\frac{1}{4}u(u-2v)(7u^2-20uv-4v^2),\] and so there are four singular members over $\Qbar$:
\begin{itemize}
    \item $[u:v]=[0:1]$: this is $Q_2$, a rank $3$ cone.
    \item $[u:v]=[2:1]$: this is $2Q_1+Q_2$, a rank $3$ cone.
    \item $[u:v]=[t:1]$ with $t=\frac{10\pm 8\sqrt{2}}{7}$: two conjugate rank $3$ cones over $L=\Q(\sqrt{2})$.
\end{itemize}
Intersecting a cone with a hyperplane that does not meet the vertex produces a base conic, which parametrizes a family of ruling lines. Intersecting these lines with $C$ yields a $g^1_2$. (Recall that $g_d^r$ denotes a linear system of degree $d$ and dimension $r+1$, which in turn produces a degree $d$ map from $C$ to $\P^r$.) 

Two of these conics are already defined over $\Q$:
\begin{itemize}
    \item From $Q_2$ (with vertex $[0:0:0:1]$ and base plane $w=0$): \[C_{1}:4x^2-2xy+y^2-2yz+2z^2=0.\]
    \item From $2Q_1+Q_2$ (with vertex $[0:0:1:0]$ and base plane $z=0$): \[C_{2}:8x^2+4xy+y^2+2w^2=0.\]
\end{itemize}
It can be checked in Magma using the intrinsic \texttt{HasRationalPoint} that neither of these conics has a rational point.

So far, we have produced two of the four divisor classes and shown that they contain no rational divisors. To find the other two, we recall that a smooth quadric surface $Q\subset\P^3$ is isomorphic to $\P^1\times\P^1$ via the Segre embedding. (See Chapter I of \cite{hartshorne2013algebraic} for details.) On $\P^1\times\P^1$ there are two families of curves, $\P^1\times\{\text{pt}\}$ and $\{\text{pt}\}\times\P^1$, each of which are mapped isomorphically to a line in $\P^3$ under the Segre embedding. Thus, $Q$ contains two families of ruling lines of the quadric. Now suppose that $L\subset Q$ is one of these ruling lines. Then since $C=Q\cap Q'$ for some other member of the pencil $Q'$, we have that $C\cap L=Q'\cap L$, which has degree two by Bezout's Theorem. This means that $L\cap C$ is an effective divisor of degree $2$ on $C$, and so as $L$ varies in one ruling, we sweep out all divisors of a $g_2^1$. 

We start by determining the smooth quadrics with square discriminant, because these have rulings defined over $\Q$. From the equation $\det M(u,v)=\frac{1}{4}u(u-2v)(7u^2-20uv-4v^2)$, we set $x=u/v$ and consider the curve $y^2=x(x-2)(7x^2-20x-4)$. A calculation in \texttt{Magma} tells us that this is a rank $0$ elliptic curve with four rational points: $(0 : 0 : 1), (2 : 0 : 1), (2 : -32 : 3), (2 : 32 : 3)$. When $x=0$ and $x=2$, the corresponding quadrics are the singular ones dealt with above. Thus, we are led consider the only smooth quadric with square discriminant $Q_s:2Q_1+3Q_2=0$, which is given by $16x^2+3y^2-4yz+4z^2+2w^2=0$. Over $\overline{\Q}$, we have that $Q_s\cong\P^1\times\P^1$, and because the discriminant is square (a calculation shows that it is $16$), the rulings are in fact defined over $\Q$. Completing the square on the portion of $Q_s$ involving $y$ and $z$ gives us the equivalent form $Q_s:16x^2+3(y-\frac{2}{3}z)^2+\frac{8}{3}z^2+2w^2=0$, which visibly has no real solutions in $\P^3$, and so there are no real (or rational) lines contained in $Q_s$. Thus, we have found four distinct $g_2^1$'s on $C$ defined over $\Q$ that contain no rational effective divisors. It follows that $C$ has no quadratic points.

\end{proof}
\end{proposition}




\subsubsection{16.48.1.l.1} Now consider the genus one curve $C$ given by the intersection of the quadrics $Q_1:16x^2 + 8xy + y^2 + z^2 - zw=0$ and $Q_2:16xy + z^2 - 2zw - w^2=0$. 
We attempt the same process as the above example. Let $A_1$ and $A_2$ denote the symmetric matrices of $Q_1$ and $Q_2$, respectively. Letting $t=v/u$ (using the notation from the previous example), we arrive at $\det(A_1+tA_2)=16t(t+1)(8t^2+8t+1)$, and so $t=0,-1,$ and $\frac{-1\pm\sqrt{2}}{4}$ give us rank $3$ quadrics in the pencil, and hence degree-$2$ divisor classes on $C_{\overline{\Q}}$. The quadric corresponding to $t=0$ is $Q_1$. For $Q_1$, we have $\partial_x=32x+8y,\partial_y=8x+2y,\partial_z=2z-w,\partial_w=-z$, and setting these to zero gives $z=0,w=0,y=-4x$. It follows that the unique singular point (the vertex) is $V=[1:-4:0:0]\in\P^3(\Q)$. Now pick the rational point $b=[0:0:1:1]$ on the conic $Q_1\cap\{y=0\}$. The line through the vertex and this base point is \[\ell:[s:t]\mapsto (x,y,z,w)=(s,-4s,t,t).\] We intersect $\ell$ with $Q_2$ and get $Q_2\cap\ell=16(s)(-4s)+t^2-2t^2-t^2=-64s^2-2t^2=0$, which implies that $t^2=-32$, and so $t=\pm 4\sqrt{-2}\in\Q(\sqrt{2}).$ Taking $s=1$ and $t=4\sqrt{-2}$ gives the point $[1:-4:4\sqrt{-2}:4\sqrt{-2}]$. This point corresponds to an effective degree two rational divisor and consequently a degree two map to $\P^1$. It follows that this modular curve has infinitely many quadratic points.

\subsection{Genus greater than or equal to two}

Since a curve of genus greater than or equal to two has an infinite number of quadratic points if and only if it is either hyperelliptic or bielliptic, we consider these cases separately. 

\subsubsection{Hyperelliptic candidates} Let us discuss the hyperelliptic case. From Lemma \ref{lemma:gonalitybasechange} we know that 
the modular curve $(X_H)_{\C}$ will have $\C$-gonality either $2$ or $1$. It cannot be one because that would imply that $X_H$ has genus $0$. 

By a work of Zywina \cite{Zywina_lowgonality}, we have a list of all the congruence subgroups that are hyperelliptic. Among them the Cummins-Pauli labels of those that have genus greater than or equal to $2$ and have prime power level are given in the set 

Gonality2:=\{\texttt{8A2}, \texttt{8B2}, \texttt{8C2}, \texttt{9A2}, \texttt{9B2}, \texttt{11A2}, \texttt{13A2}, \texttt{16A2}, \texttt{16B2}, \texttt{16C2}, \texttt{16D2}, \texttt{16E2}, \texttt{16F2}, \texttt{16G2}, \texttt{16H2}, \texttt{16I2}, \texttt{16J2}, \texttt{16K2}, \texttt{16L2}, \texttt{19A2}, \texttt{23A2}, \texttt{25A2}, \texttt{25B2}, \texttt{25C2}, \texttt{25D2}, \texttt{25E2}, \texttt{25F2}, \texttt{27A2}, \texttt{27B2}, \texttt{29A2}, \texttt{31A2}, \texttt{32A2}, \texttt{32B2}, \texttt{32C2}, \texttt{37A2}, \texttt{64A2}, \texttt{8B3}, \texttt{16B3}, \texttt{16C3}, \texttt{16D3}, \texttt{16E3}, \texttt{16F3}, \texttt{16I3}, \texttt{16J3}, \texttt{16M3}, \texttt{16S3}, \texttt{32B3}, \texttt{32C3}, \texttt{32D3}, \texttt{32H3}, \texttt{32K3}, \texttt{32M3}, \texttt{41A3}, \texttt{64A3}, \texttt{25A4}, \texttt{25D4}, \texttt{32B4}, \texttt{47A4}, \texttt{16G5}, \texttt{59A5}, \texttt{71A6}, \texttt{32E7}, \texttt{64D7}\}.

For each of the congruence groups whose label is in the set Gonality2 we use Proposition \ref{prop:boundingthelevel} to get an upper bound on the level of $H$. If the level of congruence subgroup is $p^n$ for some natural number $n$, then an upper bound on the $\GL_2$ level of $H$ is summarized below.

\begin{itemize}
    \item For $p=2$, an upper bound is $512.$
    \item For $p=3$, an upper bound is $81.$
    \item For $p=5$, an upper bound is $125.$
    \item For $p \in \{11,13,19,23,29,31,37,41,47,59,71\}$, the upper bound is $p.$
\end{itemize}

For $p=2$, we have an upper bound of $512$ for congruence subgroups with labels $\{\texttt{64A3}, \texttt{64D7}\}.$ For these two groups we observe that there is no subgroup $H$ of $\GL_2(\Z_2)$ of $\GL_2$ level $512$ such that $\Gamma_H$ has label in $\{\texttt{64A3}, \texttt{64D7}\}.$ So, if there exists a $H$ such that $\Gamma_H$ has label in $\{\texttt{64A3}, \texttt{64D7}\}$, then its $\GL_2$ level is at most $256.$ These computations can be verified in the file \href{https://github.com/mjcerchia/two-adic-galois-images-quadratic/blob/main/Hyperelliptic%20prime%20power%20level%20upper%20bound%20on%20GL2%20level}{\texttt{Hyperelliptic prime power level upper bound on GL2 level} }.

The reason for lowering the level for these two cases is not theoretical but rather pragmatic because of the completeness of LMFDB database as explained in the beginning of this section. We then browse through the database and look at all the groups $H$ such that  $\Gamma_H$ has label in the set Gonality2. The set of LMFDB labels of all the candidates is available at \href{https://github.com/mjcerchia/two-adic-galois-images-quadratic/blob/main/Hyperellipticcandidates}{\texttt{Hyperellipticcandidates.} }The maximum $\GL_2$ level in this case is $71$. We provide a summary in the table below.

\begin{table}[H]
\begin{tabular}{|l|l|}
\hline
Genus & No of candidates \\ \hline
2     & 123              \\  \hline
3     & 225              \\  \hline
4     & 1                \\  \hline
5     & 1                \\  \hline
6     & 1                \\  \hline
7     & 8                \\  \hline

\end{tabular}
\caption{Number of groups of prime power level that can possibly be hyperelliptic.}
\label{table:hypcandidates}
\end{table}
\vspace*{1em}

\subsubsection{Positive rank bielliptic candidates}

Let $H$ be an open subgroup of $\GL_2(\Zhat)$ that contains $-I$ and has surjective determinant and is of prime power level. If $X_H$ is positive rank bielliptic, then from Lemma \ref{lemma:gonalitybasechange} we know that 
the modular curve $(X_H)_{\C}$ will also be positive rank bielliptic. 

By a work of Zywina \cite{Zywina_lowgonality}, we have a list of all the congruence subgroups that are bielliptic. Among them the Cummins-Pauli labels of those that have genus greater than or equal to $2$ and have prime power level and are not hyperelliptic are given in the following set

biellipticlabels:= \{\texttt{7A3}, \texttt{8A3}, \texttt{11A3}, \texttt{16A3}, \texttt{16G3}, \texttt{16H3}, \texttt{16K3}, \texttt{16L3}, \texttt{16N3}, \texttt{16O3}, \texttt{16P3}, \texttt{16Q3}, \texttt{16R3}, \texttt{27A3}, \texttt{32A3}, \texttt{32E3}, \texttt{32F3}, \texttt{32G3}, \texttt{32I3}, \texttt{32J3}, \texttt{32L3}, \texttt{32N3}, \texttt{32O3}, \texttt{32P3}, \texttt{32Q3}, \texttt{43A3}, \texttt{49A3}, \texttt{64B3}, \texttt{9A4}, \texttt{9B4}, \texttt{9C4}, \texttt{11A4}, \texttt{16A4}, \texttt{16B4}, \texttt{16C4}, \texttt{27A4}, \texttt{27B4}, \texttt{27C4}, \texttt{27D4}, \texttt{29A4}, \texttt{32A4}, \texttt{32C4}, \texttt{37B4}, \texttt{53A4}, \texttt{61A4}, \texttt{81A4}, \texttt{8A5}, \texttt{11A5}, \texttt{16A5}, \texttt{16B5}, \texttt{16C5}, 
    \texttt{16D5}, \texttt{16E5}, \texttt{16F5}, \texttt{16H5}, \texttt{16I5}, \texttt{16J5}, \texttt{16K5}, \texttt{16M5}, \texttt{16N5}, \texttt{16O5}, \texttt{17A5}, \texttt{32A5}, \texttt{32B5}, \texttt{32C5}, \texttt{32D5}, \texttt{32E5}, \texttt{32F5}, \texttt{32G5}, \texttt{32H5}, \texttt{32I5}, \texttt{32J5}, 
    \texttt{32K5}, \texttt{32L5}, \texttt{32M5}, \texttt{32N5}, \texttt{32O5}, \texttt{41A5}, \texttt{64A5}, \texttt{64B5}, \texttt{64C5}, \texttt{64D5}, \texttt{79A6}, \texttt{27D7}, \texttt{27E7}, \texttt{32A7}, \texttt{32G7}, \texttt{32H7}, \texttt{32K7}, \texttt{64F7}, \texttt{81A7}, \texttt{83A7}, \texttt{89A7}, \texttt{101A8}, \texttt{16E9}, \texttt{32E9}, \texttt{64A9}, \texttt{131A11}\}.

For each of the congruence groups whose label is in the set biellipticlabels we use Proposition \ref{prop:boundingthelevel} to get an upper bound on the level of $H$. If the level of congruence subgroup is $p^n$ for some natural number $n$, then an upper bound on the $\GL_2$ level of $H$ is summarized below.

\begin{itemize}
    \item For $p=2$, an upper bound is $512.$
    \item For $p=3$, an upper bound is $243.$
    \item For $p=7$, an upper bound is $343.$
    \item For $p \in \{11,17,29,37,41,43,53,61,79,83,89,101,131\}$, the upper bound is $p.$
\end{itemize}

For $p=2$, we have an upper bound of $512$ for congruence subgroups with labels $\{\texttt{64B3}, \texttt{64A5}, \texttt{64B5}, \texttt{64C5},\texttt{64A9}\}.$ For $p=7$, we have an upper bound of $343$ for congruence subgroups with label \texttt{49A3}. For these groups we perform computations to observe that there is no subgroup $H$ of $\GL_2(\Z_2)$ or $\GL_2(\Z_7)$ of $\GL_2$ level $512$ or $343$ such that $\Gamma_H$ has label in $\{\texttt{64B3}, \texttt{64A5}, \texttt{64B5}, \texttt{64C5},\texttt{64A9}, \texttt{49A3}\}.$ So, if there exists an $H$ such that $\Gamma_H$ has label in $\{\texttt{64B3}, \texttt{64A5}, \texttt{64B5}, \texttt{64C5},\texttt{64A9}\}$, then its $\GL_2$ level is at most $256.$ If there exists an $H$ such that $\Gamma_H$ has label \texttt{49A3}, then its $\GL_2$ level is at most $49.$ These computations can be verified in the file \href{https://github.com/mjcerchia/two-adic-galois-images-quadratic/blob/main/Bielliptic%20prime%20power%20level%20upper%20bound%20on%20GL2%20level}{\texttt{Bielliptic prime power level upper bound on GL2 level} }.

We then browse through the LMFDB database and look at all the groups $H$ such that $\Gamma_H$ has label in the set  biellipticlabels. The set of LMFDB labels of modular curves associated to all such $H$ is available at \href{https://github.com/mjcerchia/two-adic-galois-images-quadratic/blob/main/biellipticlabels%20twist}{\texttt{biellipticlabels twist.} }The maximum $\GL_2$ level in this case is $131$. We provide a summary in the table below.
These modular curves along with the ones that are not hyperelliptic (of genus at least two) are our pool of candidates to check whether they are positive rank bielliptic or not.

\begin{table}[H]
\begin{tabular}{|l|l|}
\hline
Genus & No of groups \\ \hline
3     & 166              \\  \hline
4     & 52                \\  \hline
5     & 527                \\  \hline
6     & 1              \\  \hline
7     & 19                \\  \hline
8     & 1              \\  \hline
9     & 76 \\  \hline
11     & 1              \\  \hline

\end{tabular}
\caption{Number of groups of prime power level that are twists of some congruence subgroup with label in the set biellipticlabels.}
\label{table:bielliptic labels twists}
\end{table}
\vspace*{1em}

\subsection{Proof of Theorem \ref{thm:boundB}} We conclude this section with a proof of Theorem \ref{thm:boundB}. It follows from the observations (noted earlier in this section) that the maximum $GL_2$-level of associated groups is 32, 49, 71 and 131 for genus 0, genus 1, hyperelliptic and positive rank bielliptic case, respectively. 

\section{Hyperelliptic Modular Curves} \label{sec:hyp}

In section \ref{sec:enumerate} we provided a list of 359 candidate modular curves that could possibly be hyperelliptic. In this section we determine which ones are hyperelliptic.

Out of 359 candidates, for 301 curves we compute their canonical model using the function \texttt{FindModelOfXG} available at \cite{Zywina_github} and then use \texttt{IsHyperelliptic} command in \texttt{MAGMA} to determine that they are hyperelliptic. Please see Section $5$ of \cite{zywina2024explicitopenimageselliptic} for details on computing canonical models of modular curves. 

We now discuss remaining 58 curves. We know from Theorem 4.3 \cite{barsx0} that curves $X_0(47)$ (\mc{47.48.4.a.1}), $X_0(59)$ (\mc{59.60.5.a.1}) and $X_0(71)$ (\mc{71.72.6.a.1}) are hyperelliptic.

For curves \mc{32.96.7.r.1} and \mc{32.96.7.s.1} we rely on LMFDB database of modular curves for its Weierstrass model and hence know that they are hyperelliptic.

There are 44 curves with labels in the set \texttt{ptlessgenus0quo} of genus bigger than or equal to $3$ that have a pointless genus 0 modular curve as a quotient. So, from Corollary \ref{cor:unique genus 0} these cannot be hyperelliptic. These computations can be verified in the file \href{https://github.com/mjcerchia/two-adic-galois-images-quadratic/blob/main/Remaining%20cases-hyperelliptic}{\texttt{Hyperellipticcandidates}}.

The geometric Weierstrass models of curves \mc{16.96.3.ex.2} and \mc{16.96.3.ez.2} are given on LMFDB.
The geometric Weierstrass model of  \mc{16.96.3.ex.2} is 
\begin{align*}
    x^4 - 2x^2yz + 4x^2z^2 - 4yz^3 + 4z^4 &= 4w^2   \\
x^2 + y^2 + z^2 &= 0
\end{align*}

There is a degree $2$ map to the pointless conic $x^2+y^2+z^2=0$ by forgetting the $w$-coordinate, so this cannot be hyperelliptic.

The geometric Weierstrass model of  \mc{16.96.3.ez.2} is 
\begin{align*}
     - 2x^2yz - 4x^2z^2 - 4yz^3 - 4z^4 &= w^2 \\
     x^2 + y^2 + z^2 &=0
\end{align*}

For this curve too, there is a degree $2$ map to the pointless conic $x^2+y^2+z^2=0$ by forgetting the $w$-coordinate, so this cannot be hyperelliptic.

For labels in \{\mc{16.48.3.j.1}, \mc{16.48.3.l.1}, \mc{16.48.3.p.1}, \mc{16.96.3.ex.1}, \mc{16.96.3.ez.1} \} we compute the automorphism group over $\Q$ and observe that there is exactly one genus $0$ quotient that is pointless. So, these cannot be hyperelliptic. 

For \mc{16.48.3.z.1} and \mc{16.48.3.s.1} we were able to observe some automorphisms. For \mc{16.48.3.z.1} we observe that composing involutions $\iota_1 \colon C \to C$ given by $[x,y,z,w,t,u] \to [x,z,y,t/2,2w,-u]$ and $\iota_2 \colon C \to C$ given by $[x,y,z,w,t,u] \to [x,z,y,w,t,-u]$ gives us an involution $\iota$ such that the quotient of $C$ by $\iota$ is genus $0$ pointless conic. For \mc{16.48.3.s.1} we observe that taking quotient by $\iota \colon C \to C$ given by $[x,y,z,w,t,u] \to [x,y,z,t/2,2w,u]$ gives a pointless genus $0$ quotient. So, from Corollary \ref{cor:unique genus 0} these cannot be hyperelliptic.

Our computations can be verified in the file \href{https://github.com/mjcerchia/two-adic-galois-images-quadratic/blob/main/Remaining%20cases-hyperelliptic}{\texttt{Remaining cases-hyperelliptic}}.

We determine if these curves are positive rank bielliptic or not in later sections of this article. In conclusion, there are 306 hyperelliptic curves of prime power level and their LMFDB labels are given in Table \ref{tab:hyp}.

\section{positive rank bielliptic curves}\label{rankone}

In this section, we outline our strategy for demonstrating that a given candidate curve is positive rank bielliptic. The LMFDB labels of the candidate curves are available at \href{https://github.com/mjcerchia/two-adic-galois-images-quadratic/blob/main/biellipticlabels%20twist}{\texttt{biellipticlabels twist}}, and we know from section \ref{sec:hyp} that these are not hyperelliptic. In total there are 896 such candidates.

Recall that we defined a curve $X$ to be \textit{positive rank bielliptic} if there is a degree two map over $\Q$ to a positive rank elliptic curve, and such a curve has an infinite number of quadratic points. Any degree two map to an elliptic curve necessarily comes from the quotient $X/\iota$ of some involution $\iota$. Now that we have limited the potential positive rank bielliptic modular curves of prime power level to a finite list, our general strategy is to determine all involutions of a candidate curve $X$, locate any genus one quotients $X/\iota$, and compute the ranks of these quotient curves (if they indeed have a rational point). The genus one quotients we obtain in this way often have complicated models, making it difficult to find rational points (which is necessary for \texttt{MAGMA} to compute the rank). In these cases, we need to first compute a simpler model of the quotient, which we do by intersecting with hyperplanes and searching for low degree divisors. If successful, we can use Riemann Roch to construct a model of the form $y^2=f_4(x)$, where $f_4$ is a degree four polynomial with rational coefficients. There are 30 curves for which we compute automorphism groups either using canonical model or singular model and find quotients that are elliptic curves of rank $1$. The folder \href{https://github.com/mjcerchia/two-adic-galois-images-quadratic/tree/main/Magma%20Code}{Magma Code} contains verifications for each one of these curves. Each curve has a separate file name corresponding to its LMFDB label.

\subsection{LMFDB deductions} The simplest way to determine that a candidate curve is positive rank bielliptic is to infer it from existing data on the LMFDB. For instance, the subgroup corresponding to the modular curve with label \mc{16.96.5.a.1} has index $96$ in $\GL_2(\Z/16\Z)$ and there is a degree two map to the modular curve with label \mc{16.48.1.de.1}, which is a rank one elliptic curve whose corresponding subgroup has index $48$ in $\GL_2(\Z/16\Z)$. It follows then that \mc{16.96.5.a.1} is positive rank bielliptic. The number of cases where such a deduction works are 140 and can be found in Table \ref{tab:LMFDB deductions}.

\subsection{Already in literature} There are some curves in our list that are already known to be positive rank bielliptic. These are $X_0(43)$(\mc{43.44.3.a.1}), $X_0(53)$(\mc{53.54.4.a.1}), $X_0(61)$(\mc{61.62.4.a.1}), $X_0(79)$(\mc{79.80.6.a.1}), $X_0(83)$(\mc{83.84.7.a.1}), $X_0(89)$(\mc{89.90.7.a.1}), $X_0(101)$(\mc{101.102.8.a.1}) and $X_0(131)$(\mc{131.132.11.a.1}). Please see Theorem 4.3 of \cite{barsx0} for these.

\subsection{Computing nicer models of quotients} In several cases, we are able to compute the automorphism group of a modular curve over $\Q$ and find all genus one quotients by an involution. However, as mentioned above, these models are typically too complicated for \texttt{MAGMA} to be able to find rational points on, and so if they are in fact elliptic curves, we cannot simply compute their ranks. If this is the case for a given modular curve, then the strategy for computing a new model of a given quotient curve $C$ is the following:

\begin{itemize}
    \item Search for divisors on $C$ by intersecting with hyperplanes. In practice, a homogeneous coordinate function often works, but we also iterate over hyperplanes whose coefficients lie in $\{0,-1,1\}$. 
    \item If we are lucky, we find a degree two divisor $D$, from which we can build a model of the form $y^2=f_4(x)$, where $f_4$ is a quartic polynomial in $x$. We do this by using functions in the Riemann-Roch spaces for $D$ and $2D$ to construct a map to the weighted projective space $\P(2,1,1)$. Such a model is helpful to us because it is much faster to search for rational points in \texttt{MAGMA}. It is also of the form that allows us to check for local solubility with the \texttt{MAGMA} intrinsic \texttt{HasPointsEverywhereLocally}.
    \item Any time we use this method but cannot find a degree two divisor, we are able to find a degree three divisor. The only difference is that we now find three functions in the Riemann-Roch space for $D$ which gives us a map to $\P^2$. The resulting model is simpler than the original but more complicated than one of the form $y^2=f_4(x)$. 
\end{itemize}
\begin{example}
    We show that the non-hyperelliptic modular curve of level $32$ with LMFDB label \href{https://beta.lmfdb.org/ModularCurve/Q/32.96.5.m.1}{32.96.5.m.1} corresponding to the subgroup $H$ of $\GL_2(\Z/32\Z)$ generated by the matrices $\begin{bmatrix}1&12\\8&31\end{bmatrix}$,$\begin{bmatrix}11&23\\ 16&1\end{bmatrix}$, $\begin{bmatrix}13&9\\16&31\end{bmatrix}$, and $\begin{bmatrix}25&3\\0&3\end{bmatrix}$ has an infinite number of quadratic points. In particular, we show that $X_H$ is positive rank bielliptic. 
    
    After using \texttt{MAGMA} to compute the automorphism group of $X_H$ over $\Q$, we determine that there are three genus one quotients by some involution. Call these curves $C_1,C_2$, and $C_3$. The Jacobian of $X_H$ has exactly one rank 1 elliptic curve factor, which we will denote by $E$. Over $\F_5$, $C_1$ and $C_3$ have a different number of points than $E$, which means that the Jacobian of neither $C_1$ nor $C_3$ could be be isogenous to $E$ since they have different $L$-functions (which is an isogeny invariant). Consequently, neither $C_1$ nor $C_3$ are rank 1. The model for $C_2$ is given in $\P^7$ by the zero locus of the following $20$ quadrics: 
\begin{align*}
Q_1&=x_1^2 + 2x_3^2 - 8x_6^2 + x_7^2 - 8x_5x_8,\\
Q_2&=x_1x_2 - x_6^2 + x_8^2,\\
Q_3&=8x_1x_4 - x_7^2 + 8x_5x_8,\\
Q_4&=x_1x_5 + 8x_4x_5 - x_3x_7,\\
Q_5&=4x_1x_6 - 32x_2x_6 + x_5x_7,\\
Q_6&=2x_3x_5 + x_1x_7 - 8x_6x_8,\\
Q_7&=-x_6x_7 + x_1x_8 + 8x_4x_8,\\
Q_8&=32x_2^2 - 4x_6^2 - x_5x_8 + 4x_8^2,\\
Q_9&=8x_2x_3 - x_7x_8,\\
Q_{10}&=8x_2x_4 - x_8^2,\\
Q_{11}&=x_2x_5 - x_4x_8,\\
Q_{12}&=x_2x_7 - x_6x_8,\\
Q_{13}&=-2x_4x_5 - x_6x_7 + 8x_2x_8 + 8x_4x_8,\\
Q_{14}&=8x_3x_4 - x_5x_7,
\end{align*}

\begin{align*}
Q_{15}&=8x_3x_6 - x_7^2,\\
Q_{16}&=-x_5x_6 + x_3x_8,\\
Q_{17}&=8x_4^2 - x_5x_8,\\
Q_{18}&=8x_4x_6 - x_7x_8,\\
Q_{19}&=-x_5x_6 + x_4x_7,\\
Q_{20}&=2x_5^2 + x_7^2 - 8x_5x_8 - 8x_8^2
\end{align*}

    This model is too complicated for \texttt{MAGMA} to find a rational point on (if one exists). Intersecting with the hyperplane $X_1=0$, we find the rational point $(0:-1/4:-2:-1/2:2:-1:4:1)$, from which we are able to use Riemann-Roch to construct the model $y^2 - 131072y = x^3 + 4096x^2 + 6291456x$ for $C_3$, which has rank $1$. This is indeed isogenous to the rank one elliptic curve factor of the Jacobian, $E$, which has a model given by $y^2 = x^3 + x^2 + x + 1$. It follows that $X_H$ is positive rank bielliptic and therefore has an infinite number of quadratic points. The \texttt{MAGMA} file \href{https://github.com/mjcerchia/two-adic-galois-images-quadratic/blob/main/Magma%20Code/32-96-5-m-1.m}{32-96-5-m-1.m} verifies these claims. 
\end{example}  

\begin{example} 
We conclude this section with the example of \mc{27.108.7.g.1}. Since this has genus $7$ we know from Theorem \ref{thm:CSineq} that up to isomorphism it has at most one genus 1 quotient. Since this is geometrically bielliptic, from Lemma 5 \cite{MR1055774} we know that it has one genus 1 quotient defined over $\Q.$ Let $C$ be that unique genus 1 quotient. There are three elliptic curves in Jacobian decomposition of \mc{27.108.7.g.1} that correspond to newforms \texttt{27.2.a.a}, \texttt{243.2.a.a} and \texttt{243.2.a.b}. The map to \texttt{27.2.a.a} is of degree 3 and it is given on LMFDB. Working modulo 7 we observe that Jacobian of $C$ must be isogenous to \texttt{243.2.a.a} which is the rank $1$ factor. So, if we can show that $C$ has a rational point, then we are done.

It takes a long time to compute its automorphism group over $\Q$ so we compute its automorphisms over $F_{17}$ and lift it to $\Q$ using the script \href{https://users.wfu.edu/rouseja/2adic/autocompute.txt}{\texttt{autocompute.txt}} due to Rouse and Zureick-Brown which is part of the code associated with \cite{rzb}. We then compute its quotient which is a genus 1 curve, intersect it with a hyperplane to find a point. This computation can be verified in the file \href{https://github.com/mjcerchia/two-adic-galois-images-quadratic/blob/main/Magma%20Code/27-108-7-g-1.m}{\texttt{27-108-7-g-1.m}}.
\end{example}

The labels of positive rank bielliptic curves that are not hyperelliptic are given in Tables \ref{tab:LMFDB deductions} and \ref{tab:posrankbi}.

\section{Non positive rank bielliptic curves}\label{nonrankone}

To show that a modular curve $X_H$ is not positive rank bielliptic, we must demonstrate that there cannot be a degree two map to a positive rank elliptic curve. The simplest approach is to compute automorphism group over $\Q$ and show either there are no genus one quotients by an involution, which would mean that the curve is not bielliptic, let alone positive rank bielliptic or show that all genus 1 quotients have finitely many points. This is not always computationally feasible. Often the automorphism groups are difficult to compute, and so in these cases we instead reduce our curve mod a prime $p$ of good reduction and show that the reduced curve fails to be bielliptic or genus 1 quotients do not match with positive rank elliptic curve factor given in decomposition of Jacobian. Even when we are able to compute automorphism group sometimes we have to simplify models to deduce whether a given genus 1 curve has finitely many points.

\subsection{LMFDB Deductions} In some cases LMFDB has genus 1 quotients that correspond to positive rank elliptic curve factors in decomposition of Jacobian so we can directly observe whether the curve is positive rank bielliptic or not. Consider for example, \mc{16.96.5.bd.1}. There is exactly one rank 1 newform \texttt{128.2.a.a} in decomposition of its Jacobian. From lemma \ref{lemma:nfactors} it suffices to find one pointless genus 1 curve whose Jacobian is isogenous to \texttt{128.2.a.a}. The curve \mc{16.48.1.cw.1} satisfies this property. The curves for which this argument works are \{\mc{16.96.5.bd.4}, \mc{16.96.5.cd.1}, \mc{16.96.5.cd.2}, \mc{16.96.5.h.1}, \mc{16.96.5.h.2}, \mc{16.96.5.a.2}, \mc{16.96.5.bs.2}, \mc{16.96.5.ch.1}, \mc{16.96.5.dk.1}, \mc{16.96.5.dp.1}, \mc{16.96.5.x.1}\}. So none of these are positive rank bielliptic. For curves with labels in \{\mc{16.96.5.dm.1}, \mc{16.96.5.dt.1}\} there are two rank 1 newform factors \texttt{256.2.a.a}, \texttt{256.2.a.b}. For both of these there is one pointless genus 1 quotient given on LMFDB whose Jacobian is isogenous to \texttt{256.2.a.a}. We now show that there is no genus 1 quotient whose Jacobian is isogenous to \texttt{256.2.a.b}. We compute all the automorphisms modulo $5$ and observe that any genus 1 quotient has either 4 or 6 points. If $E$ is isogenous to \texttt{256.2.a.b} then it must have 10 points modulo 5. So, these two curves are not positive rank bielliptic. Our computations can be verified in the files \href{https://github.com/mjcerchia/two-adic-galois-images-quadratic/tree/main/Not%20positive%20rank}{\texttt{16-96-5-dm-1.m}} and \href{https://github.com/mjcerchia/two-adic-galois-images-quadratic/tree/main/Not%20positive%20rank}{\texttt{16-96-5-dt-1.m}}.

The curve \mc{16.96.3.fc.1} has a degree 2 map to \mc{16.48.1.bx.1} whose Jacobian is isogenous to \texttt{256.2.a.a}. The curve \mc{16.48.1.bx.1} is pointless and is the only genus 1 quotient that corresponds to \texttt{256.2.a.a}. So there does not exist a degree 2 map to the positive rank elliptic curve in its Jacobian. Hence this curve is not positive rank bielliptic. This argument is also applicable to the curve \mc{16.96.3.fj.1}. This can be verified in the file \href{https://github.com/mjcerchia/two-adic-galois-images-quadratic/blob/main/biellipticlabels%20twist}{\texttt{Section 7.1.m}}.

For the curves \mc{16.96.3.cz.1}, \mc{16.96.3.dm.1},  \mc{16.96.3.bf.1} and \mc{16.96.3.bo.1} we observe that the group $\Aut(X_{\Gamma_H},\pi_{\Gamma_H})$ (please see section \ref{sec:modularcurves} for this notation) is of order 32. From Table 1 \cite{MR2182037} we know that the order of automorphism group over $\C$ must be 32. So, the automorphism group over $\Q$ is the subgroup of $\Aut(X_{\Gamma_H},\pi_{\Gamma_H})$ defined over $\Q$. We observe that the order of this subgroup in all of these cases is 8, hence there can be maximum $7$ involutions, and all $7$ quotients are given on LMFDB. The genus 1 quotients corresponding to rank 1 newform factors are pointless. So, none of these are positive rank bielliptic. This can be verified in the file \href{https://github.com/mjcerchia/two-adic-galois-images-quadratic/blob/main/biellipticlabels%20twist}{\texttt{Section 7.1.m}}.

From Theorem \ref{thm:CSineq} we know that a curve of genus at least 6 must have at most one genus 1 quotient by an involution. For curves with labels in the set \texttt{g9} (available at github) their genus 1 quotient is given on LMFDB and their Jacobians have rank $0$. So, these curves cannot be positive rank bielliptic. This can be verified in the file \href{https://github.com/mjcerchia/two-adic-galois-images-quadratic/blob/main/biellipticlabels%20twist}{\texttt{Section 7.1.m}}.

From Lemma \ref{lemma:jacfactor} we know that if the Jacobian has rank $0$ or does not have any positive rank elliptic curve in its decomposition, then it cannot be positive rank bielliptic. The number of such candidates is 537 and this list is available at \href{https://github.com/mjcerchia/two-adic-galois-images-quadratic/blob/main/biellipticlabels%20twist}{\texttt{biellipticlabels twist}}.

\subsection{Already in literature} From Theorem 4.3 \cite{barsx0}, the modular curve $X_0(81)$(\mc{81.108.4.a.1}) is not positive rank bielliptic.

\subsection{Not bielliptic} If a curve $X$ is not bielliptic, it cannot be positive rank bielliptic. If we can determine all the involutions of a curve $X$ (which we usually obtain by first computing the automorphism group of $X$) we can form quotients of $X$ by each of these involutions. If none of these are genus $1$, then there cannot be a degree two map from $X$ to an elliptic curve, and so we are done. 

\begin{example}
    Consider the non-hyperelliptic modular curve with LMFDB label \mc{16.96.5.dy.1}. A canonical model of this curve in $\P^4$ is defined by $3$ equations, and we find using \texttt{MAGMA} that there are no genus one quotients by involutions. Consequently, this modular curve is not bielliptic, and so it has a finite number of quadratic points. See the file \href{https://github.com/mjcerchia/two-adic-galois-images-quadratic/blob/main/Not%20positive%20rank/16-96-5-dy-1.m}{16-96-5-dy-1.m} for verification. 
\end{example}

Often it is slow to compute the automorphism group of a curve over $\Q$, and for such a curve $C$ we try to reduce the curve modulo a prime of good reduction and then compute the automorphism group of the reduced curve $C_p$ over the corresponding finite field. If this is successful and there are no genus one quotients, then $C$ cannot be positive rank bielliptic. 

\begin{example}
Consider the nonhyperelliptic curve with LMFDB label \mc{16.96.5.ec.1} corresponding to subgroup $H$ of $\GL_2(\Z/16\Z)$ generated by the matrices  $\begin{bmatrix}3&6\\14&13
    \end{bmatrix}$,$\begin{bmatrix}1&4\\14&15\end{bmatrix}$, $\begin{bmatrix}11&14\\14&5\end{bmatrix}$, and $\begin{bmatrix}13&15\\6&7\end{bmatrix}$. It is difficult to compute the automorphism group of this curve over $\Q$, but we are able to reduce the curve mod $3$ and compute the automorphism group of this reduced curve. We find there are no genus one quotients by involutions, and so the curve cannot be bielliptic over $\Q$ either. See the file \href{https://github.com/mjcerchia/two-adic-galois-images-quadratic/blob/main/Not%20positive%20rank/16-96-5-ec-1.m}{16-96-5-ec-1.m} for verification.

\end{example}

Even if there are genus one quotients by an involution of the reduction of a modular curve, we come sometimes rule out that such a curve is positive rank bielliptic by comparing the number of points on the quotient curves to any rank one elliptic curve factors in the Jacobian decomposition, as in the next example. 

\begin{example} \label{ex:mismatch}
    Consider the non-hyperelliptic modular curve with LMFDB label 16.192.5.bs.1. Over $\F_5$, there is one genus one quotient by involution, which has a different number of points than the single rank one elliptic curve factor of the Jacobian of the modular curve. Consequently, this curve is not positive rank bielliptic and so it has a finite number of quadratic points. See \href{https://github.com/mjcerchia/two-adic-galois-images-quadratic/blob/main/Not%20positive%20rank/16-192-5-bs-1.m}{16-192-5-bs-1.m} for verification. 
\end{example}

\subsection{Pointless genus one quotient} If on the other hand there exists a degree two map from a modular curve $X$ to a genus $1$ curve $C$, where $C\cong X/\iota$ for some involution $\iota$ of $X$, it might be the case that $C$ does not have a rational point, which would mean that it is not an elliptic curve over $\Q$. To determine this, we check local solubility. If no rational point exists, then the map $X\to X/\langle\iota\rangle\to \operatorname{Jac}(X/\langle\iota\rangle)$ has degree greater than two. Note that $\operatorname{Jac}(X/\langle\iota\rangle)$ -- the Jacobian of $X/\iota$ -- is an elliptic curve appearing in the Jacobian decomposition of $X$. For some genus $5$ curves, using lemma \ref{lemma:nfactors} if an elliptic curve $E$ appears $n$ number of times in the Jacobian decomposition of $X$, then we find $n$ pointless genus 1 curves that come from $n$ different involutions whose Jacobians are isogenous to $E$. In all such cases $n$ is either one or two.

\begin{example}
The non-hyperelliptic modular curve of level $32$ with LMFDB label \mc{32.96.5.f.1} corresponding to subgroup $H$ of $\GL_2(\Z/32\Z)$ generated by the matrices  $\begin{bmatrix}7&5\\4&7
    \end{bmatrix}$,$\begin{bmatrix}29&24\\8&7\end{bmatrix}$, and $\begin{bmatrix}29&29\\4&5\end{bmatrix}$ has a finite number of quadratic points. In particular, we show that it is not positive rank bielliptic.
A canonical model for $X_H$ is given in $\P^4$ by the zero locus of the following polynomials: \begin{align*} 
2x^2 +wt,\\ 
2yz - w^2,\\
2y^2+32z^2+t^2.
\end{align*} We determine the automorphism group of this curve over $\Q$ with \texttt{MAGMA} and identify any involutions. We then find that there are three genus one quotients by involutions, call them $C_1$, $C_2$, and $C_3$. Consequently, $X_H$ is bielliptic (but not necessarily positive rank bielliptic). There is exactly one rank one elliptic curve factor in the Jacobian decomposition of $X_H$, and it has the model over $\Q$ given by $E:y^2 = x^3 - 2x$. This means there is some map from $X_H$ to a rank one elliptic curve (namely $E$); it remains to be shown that this map cannot be degree two. The models for each of the quotients are complicated -- $20$ equations in $\P^7$. Over $\F_5$, $C_1$ and $C_2$ have a different number of points than $E$, and so we do not have to consider them, since this implies that their Jacobians could not be isogenous to $E$. The model for $C_3$ is given in $\P^7$ by the zero locus of the following $20$ quadrics:
\begin{align*}
Q_1&=2048x_1^2 + 16x_3^2 + x_5x_8,\\ 
Q_2&=512x_1x_2 + x_5x_8,\\ 
Q_3&=16x_1x_4 - 4x_7^2 - x_8^2,\\ 
Q_4&=4x_1x_5 - 4x_3x_7 - x_4x_8,\\ 
Q_5&=512x_1x_6 + x_5x_7,\\ 
Q_6&=x_3x_5 + 128x_1x_7,\\ 
Q_7&=x_4x_5 + 128x_1x_8,\\ 
Q_8&=128x_2^2 + 512x_6^2 + x_5x_8,\\ 
Q_9&=4x_2x_3 - x_7x_8,\\ 
Q_{10}&=4x_2x_4 - x_8^2,\\ 
Q_{11}&=x_2x_5 - x_4x_8,\\ 
Q_{12}&=x_2x_7 - x_6x_8,\\ 
Q_{13}&=x_4x_5 + 128x_6x_7 + 32x_2x_8,\\ 
Q_{14}&=4x_3x_4 - x_5x_7,\\ 
Q_{15}&=4x_3x_6 - x_7^2,\\ 
Q_{16}&=-x_5x_6 + x_3x_8,\\ 
Q_{17}&=4x_4^2 - x_5x_8,\\ 
Q_{18}&=4x_4x_6 - x_7x_8,\\ 
Q_{19}&=-x_5x_6 + x_4x_7,\\ 
Q_{20}&=x_5^2 + 128x_7^2 + 32x_8^2
\end{align*}
$C_3$ has the same number of points as $E$ over $\F_p$ for primes up to $10000$, but \texttt{MAGMA} cannot find a rational point (and therefore cannot directly compute the rank of $C_3$ as an elliptic curve). To overcome this, we build a simpler model, which we do by intersecting $C_3$ with all hyperplanes with coefficients in $\{-1,0,1\}$, and this produces a degree two divisor. This allows us to construct a map from $C$ into the weighted projective space $\P(2,1,1)$, which gives us a model for $C_3$ of the form $y^2 = -(65536x^4 + 128)$. We find that this fails local solubility, and hence $C_3$ has no rational points. Using \texttt{MAGMA}, we find that Jacobian of this curve is $y^2 = 4x^3 - 8388608x$, which is isogenous to $E$. Since $C_3$ has no rational points, the map to its Jacobian has degree higher than $1$, which means that it is impossible to have a degree two map to a rank one elliptic curve. As $X_H$ is not hyperelliptic, it must have a finite number of points. The \texttt{MAGMA} file verifying this can be found here \href{https://github.com/mjcerchia/two-adic-galois-images-quadratic/blob/main/Not%20positive%20rank/32-96-5-f-1.m}{32-96-5-f-1.m}.   
\end{example}

\begin{remark}
    In this case, we were somewhat lucky in that reducing by a single prime allowed us to dismiss two of the three quotient curves. In other similar cases, we need to use different primes for different quotient curves. See, for instance, \href{https://github.com/mjcerchia/two-adic-galois-images-quadratic/blob/main/Not%20positive%20rank/16-96-5-l-1.m}{16-96-5-l-1.m} or \href{https://github.com/mjcerchia/two-adic-galois-images-quadratic/blob/main/Not%20positive%20rank/32-96-5-bf-1.m}{32-96-5-bf-1.m}.
\end{remark}

\newpage

\section{Appendix: Summary of Cases}

\begin{longtable}[H]{|p{2.5cm}|p{2.5cm}|p{2.5cm}|p{2.5cm}|p{2.5cm}|p{2.5cm}|}
\caption{LMFDB labels of genus 0 prime power level modular curves.\label{tab:genus 0}}\\
 \hline
\hline
 \multicolumn{6}{| p{2.5cm} |}{}\\
 
\mc{2.2.0.a.1} & \mc{2.3.0.a.1} & \mc{2.6.0.a.1} & \mc{3.3.0.a.1} & \mc{3.4.0.a.1} & \mc{3.6.0.a.1} \\  \hline
\mc{3.6.0.b.1} & \mc{3.12.0.a.1} & \mc{4.2.0.a.1} & \mc{4.4.0.a.1} & \mc{4.6.0.a.1} & \mc{4.6.0.b.1} \\  \hline

\mc{4.6.0.c.1} & \mc{4.6.0.d.1} & \mc{4.6.0.e.1} & \mc{4.8.0.a.1} &\mc{4.8.0.b.1} & \mc{4.12.0.a.1} \\ \hline
 \mc{4.12.0.d.1} & \mc{4.12.0.e.1} & \mc{4.12.0.f.1} & \mc{4.24.0.a.1} & \mc{4.12.0.b.1} & \mc{4.12.0.c.1} \\ \hline

\mc{4.24.0.b.1} & \mc{4.24.0.c.1} & \mc{5.5.0.a.1} & \mc{5.6.0.a.1} &\mc{5.10.0.a.1} & \mc{5.12.0.a.1} \\ \hline
\mc{5.20.0.a.1} & \mc{5.20.0.b.1} & \mc{5.30.0.a.1} & \mc{5.30.0.b.1}  & \mc{5.12.0.a.2} & \mc{5.15.0.a.1} \\ \hline
\mc{5.60.0.a.1} & \mc{5.60.0.b.1} & \mc{7.8.0.a.1} & \mc{7.21.0.a.1} & \mc{7.24.0.a.1} & \mc{7.24.0.a.2} \\ \hline
\mc{8.2.0.a.1} & \mc{8.2.0.b.1} & \mc{8.6.0.a.1} & \mc{8.6.0.b.1}& \mc{7.24.0.b.1} & \mc{7.28.0.a.1} \\ \hline
\mc{8.6.0.c.1} & \mc{8.6.0.d.1} & \mc{8.6.0.e.1} & \mc{8.6.0.f.1}& \mc{8.8.0.a.1} & \mc{8.8.0.b.1} \\ \hline
\mc{8.12.0.c.1} & \mc{8.12.0.d.1} & \mc{8.12.0.e.1} & \mc{8.12.0.f.1}& \mc{8.12.0.a.1} & \mc{8.12.0.b.1} \\ \hline
\mc{8.12.0.g.1} & \mc{8.12.0.h.1} & \mc{8.12.0.i.1} & \mc{8.12.0.j.1}&\mc{8.12.0.k.1} & \mc{8.12.0.l.1}  \\ \hline
\mc{8.12.0.o.1} & \mc{8.12.0.p.1} & \mc{8.12.0.q.1} & \mc{8.12.0.r.1}& \mc{8.12.0.m.1} & \mc{8.12.0.n.1} \\ \hline

\mc{8.12.0.s.1} & \mc{8.12.0.t.1} & \mc{8.12.0.u.1} & \mc{8.12.0.v.1} &\mc{8.12.0.w.1} & \mc{8.12.0.x.1}\\ \hline
\mc{8.16.0.a.1} & \mc{8.24.0.a.1} & \mc{8.24.0.b.1} & \mc{8.24.0.ba.1} & \mc{8.12.0.y.1} & \mc{8.12.0.z.1} \\ \hline

\mc{8.24.0.ba.2} & \mc{8.24.0.bb.1} & \mc{8.24.0.bb.2} & \mc{8.24.0.bc.1} &\mc{8.24.0.bd.1} & \mc{8.24.0.be.1} \\ \hline
\mc{8.24.0.bh.1} & \mc{8.24.0.bi.1} & \mc{8.24.0.bj.1} & \mc{8.24.0.bk.1} & \mc{8.24.0.bf.1} & \mc{8.24.0.bg.1} \\ \hline
 
\mc{8.24.0.bk.2} & \mc{8.24.0.bl.1} & \mc{8.24.0.bl.2} & \mc{8.24.0.bm.1}&\mc{8.24.0.bn.1} & \mc{8.24.0.bo.1} \\ \hline
\mc{8.24.0.br.1} & \mc{8.24.0.bs.1} & \mc{8.24.0.bt.1} & \mc{8.24.0.c.1} & \mc{8.24.0.bp.1} & \mc{8.24.0.bq.1} \\ \hline
 
\mc{8.24.0.d.1} & \mc{8.24.0.d.2} & \mc{8.24.0.e.1} & \mc{8.24.0.e.2}&\mc{8.24.0.f.1} & \mc{8.24.0.g.1} \\ \hline
\mc{8.24.0.j.1} & \mc{8.24.0.k.1} & \mc{8.24.0.l.1} & \mc{8.24.0.m.1} & \mc{8.24.0.h.1} & \mc{8.24.0.i.1} \\ \hline
 
\mc{8.24.0.n.1} & \mc{8.24.0.o.1} & \mc{8.24.0.p.1} & \mc{8.24.0.q.1}&\mc{8.24.0.r.1} & \mc{8.24.0.s.1} \\ \hline
\mc{8.24.0.v.1} & \mc{8.24.0.w.1} & \mc{8.24.0.x.1} & \mc{8.24.0.y.1} & \mc{8.24.0.t.1} & \mc{8.24.0.u.1} \\ \hline

\mc{8.24.0.z.1} & \mc{8.48.0.a.1} & \mc{8.48.0.b.1} & \mc{8.48.0.b.2}&\mc{8.48.0.c.1} & \mc{8.48.0.d.1}  \\ \hline
\mc{8.48.0.f.1} & \mc{8.48.0.g.1} & \mc{8.48.0.h.1} & \mc{8.48.0.h.2}& \mc{8.48.0.e.1} & \mc{8.48.0.e.2} \\ \hline

\mc{8.48.0.i.1} & \mc{8.48.0.j.1} & \mc{8.48.0.j.2} & \mc{8.48.0.k.1} & \mc{8.48.0.k.2} & \mc{8.48.0.l.1} \\ \hline
\mc{8.48.0.m.2} & \mc{8.48.0.n.1} & \mc{8.48.0.n.2} & \mc{8.48.0.o.1} & \mc{8.48.0.l.2} & \mc{8.48.0.m.1} \\ \hline
 
\mc{8.48.0.p.1} & \mc{8.48.0.q.1} & \mc{8.48.0.q.2} & \mc{9.9.0.a.1}& \mc{9.12.0.a.1} & \mc{9.12.0.b.1}\\ \hline
\mc{9.18.0.c.1} & \mc{9.18.0.d.1} & \mc{9.27.0.a.1} & \mc{9.27.0.b.1} & \mc{9.18.0.a.1} & \mc{9.18.0.b.1} \\ \hline
 
\mc{9.36.0.a.1} & \mc{9.36.0.b.1} & \mc{9.36.0.c.1} & \mc{9.36.0.d.1}&\mc{9.36.0.d.2} & \mc{9.36.0.e.1}  \\ \hline
\mc{9.36.0.g.1} & \mc{13.14.0.a.1} & \mc{13.28.0.a.1} & \mc{13.28.0.a.2}& \mc{9.36.0.f.1} & \mc{9.36.0.f.2} \\ \hline
 
\mc{13.42.0.a.1} & \mc{13.42.0.a.2} & \mc{13.42.0.b.1} & \mc{16.16.0.a.1} &\mc{16.16.0.b.1} & \mc{16.24.0.a.1}\\ \hline
\mc{16.24.0.d.1} & \mc{16.24.0.e.1} & \mc{16.24.0.e.2} & \mc{16.24.0.f.1} & \mc{16.24.0.b.1} & \mc{16.24.0.c.1} \\ \hline
 
\mc{16.24.0.f.2} & \mc{16.24.0.g.1} & \mc{16.24.0.h.1} & \mc{16.24.0.i.1}&\mc{16.24.0.j.1} & \mc{16.24.0.k.1} \\ \hline
\mc{16.24.0.l.2} & \mc{16.24.0.m.1} & \mc{16.24.0.m.2} & \mc{16.24.0.n.1} & \mc{16.24.0.k.2} & \mc{16.24.0.l.1} \\ \hline
 
\mc{16.24.0.n.2} & \mc{16.24.0.o.1} & \mc{16.24.0.o.2} & \mc{16.24.0.p.1} &\mc{16.24.0.p.2} & \mc{16.48.0.a.1}\\ \hline
\mc{16.48.0.ba.2} & \mc{16.48.0.bb.1} & \mc{16.48.0.bb.2} & \mc{16.48.0.c.1} & \mc{16.48.0.b.1} & \mc{16.48.0.ba.1} \\ \hline

\mc{16.48.0.c.2} & \mc{16.48.0.d.1} & \mc{16.48.0.d.2} & \mc{16.48.0.e.1} &\mc{16.48.0.f.1} & \mc{16.48.0.g.1}\\ \hline
\mc{16.48.0.i.1} & \mc{16.48.0.j.1} & \mc{16.48.0.k.1} & \mc{16.48.0.l.1}  & \mc{16.48.0.h.1} & \mc{16.48.0.h.2} \\ \hline

\mc{16.48.0.l.2} & \mc{16.48.0.m.1} & \mc{16.48.0.m.2} & \mc{16.48.0.n.1} &\mc{16.48.0.o.1} & \mc{16.48.0.p.1}\\ \hline
\mc{16.48.0.s.1} & \mc{16.48.0.t.1} & \mc{16.48.0.t.2} & \mc{16.48.0.u.1} & \mc{16.48.0.q.1} & \mc{16.48.0.r.1} \\ \hline
 
\mc{16.48.0.u.2} & \mc{16.48.0.v.1} & \mc{16.48.0.v.2} & \mc{16.48.0.w.1}&\mc{16.48.0.w.2} & \mc{16.48.0.x.1}  \\ \hline
\mc{16.48.0.y.2} & \mc{16.48.0.z.1} & \mc{16.48.0.z.2} & \mc{25.30.0.a.1}& \mc{16.48.0.x.2} & \mc{16.48.0.y.1} \\ \hline
 
\mc{25.60.0.a.1} & \mc{25.60.0.a.2} & \mc{27.36.0.a.1} & \mc{32.32.0.a.1}&\mc{32.32.0.b.1} & \mc{32.48.0.a.1} \\ \hline
\mc{32.48.0.d.1} & \mc{32.48.0.e.1} & \mc{32.48.0.e.2} & \mc{32.48.0.f.1} & \mc{32.48.0.b.1} & \mc{32.48.0.c.1} \\ \hline
 
\mc{32.48.0.f.2} & & & & \\ \hline

\end{longtable}

\begin{longtable}[H]{|p{2.5cm}|p{2.5cm}|p{2.5cm}|p{2.5cm}|p{2.5cm}|p{2.5cm}|}
\caption{LMFDB labels of genus 1 prime power level modular curves that have infinitely many quadratic points.\label{tab:genus 1}}\\
 \hline
\hline
 \multicolumn{6}{| p{2.5cm} |}{}\\

\mc{7.42.1.b.1} & \mc{7.56.1.a.1} & \mc{7.56.1.b.1} & \mc{7.84.1.a.1} & \mc{8.12.1.a.1} & \mc{8.12.1.b.1} \\ \hline
\mc{8.12.1.c.1} & \mc{8.12.1.d.1} & \mc{8.24.1.a.1} & \mc{8.24.1.b.1} &
\mc{8.24.1.ba.1} & \mc{8.24.1.bb.1} \\ \hline
\mc{8.24.1.bd.1} & \mc{8.24.1.be.1} & \mc{8.24.1.bf.1} & \mc{8.24.1.c.1} &
\mc{8.24.1.d.1} & \mc{8.24.1.f.1} \\ \hline
\mc{8.24.1.l.1} & \mc{8.24.1.m.1} & \mc{8.24.1.n.1} & \mc{8.24.1.o.1} &
\mc{8.24.1.r.1} & \mc{8.24.1.s.1} \\ \hline
\mc{8.24.1.t.1} & \mc{8.24.1.u.1} & \mc{8.24.1.w.1} & \mc{8.24.1.x.1} &
\mc{8.24.1.y.1} & \mc{8.48.1.bb.1} \\ \hline
\mc{8.48.1.bc.1} & \mc{8.48.1.bn.1} & \mc{8.48.1.bp.1} & \mc{8.48.1.bs.1} &
\mc{8.48.1.bv.1} & \mc{8.48.1.c.1} \\ \hline
\mc{8.48.1.g.1} & \mc{8.48.1.g.2} & \mc{8.48.1.h.1} & \mc{8.48.1.h.2} &
\mc{8.48.1.i.1} & \mc{8.48.1.i.2} \\ \hline
\mc{8.48.1.k.1} & \mc{8.48.1.k.2} & \mc{8.48.1.m.1} & \mc{8.48.1.m.2} &
\mc{8.48.1.n.1} & \mc{8.48.1.p.1} \\ \hline
\mc{8.48.1.t.1} & \mc{8.96.1.d.1} & \mc{8.96.1.e.2} & \mc{8.96.1.f.2} &
\mc{8.96.1.g.1} & \mc{8.96.1.g.2} \\ \hline
\mc{8.96.1.j.1} & \mc{8.96.1.l.1} & \mc{8.96.1.m.1} & \mc{9.12.1.a.1} &
\mc{9.36.1.a.1} & \mc{9.36.1.b.1} \\ \hline
\mc{9.36.1.b.2} & \mc{9.36.1.c.1} & \mc{9.54.1.a.1} & \mc{9.108.1.a.1} &
\mc{9.108.1.a.2} & \mc{9.108.1.b.1} \\ \hline
\mc{11.12.1.a.1} & \mc{11.55.1.a.1} & \mc{11.55.1.b.1} & \mc{11.60.1.a.1} &
\mc{11.60.1.a.2} & \mc{11.60.1.b.1} \\ \hline
\mc{11.60.1.b.2} & \mc{11.60.1.c.1} & \mc{16.24.1.a.1} & \mc{16.24.1.b.1} &
\mc{16.24.1.c.1} & \mc{16.24.1.d.1} \\ \hline
\mc{16.24.1.e.1} & \mc{16.24.1.e.2} & \mc{16.24.1.f.1} & \mc{16.24.1.f.2} &
\mc{16.24.1.g.1} & \mc{16.24.1.g.2} \\ \hline
\mc{16.24.1.h.1} & \mc{16.24.1.h.2} & \mc{16.24.1.i.1} & \mc{16.24.1.j.1} &
\mc{16.24.1.k.1} & \mc{16.24.1.l.1} \\ \hline
\mc{16.24.1.m.1} & \mc{16.24.1.m.2} & \mc{16.24.1.n.1} & \mc{16.24.1.n.2} &
\mc{16.48.1.a.1} & \mc{16.48.1.a.2} \\ \hline
\mc{16.48.1.b.1} & \mc{16.48.1.b.2} & \mc{16.48.1.ba.1} & \mc{16.48.1.bd.1} &
\mc{16.48.1.bf.1} & \mc{16.48.1.bg.1} \\ \hline
\mc{16.48.1.bl.1} & \mc{16.48.1.bm.1} & \mc{16.48.1.bn.1} & \mc{16.48.1.bo.1} &
\mc{16.48.1.bp.1} & \mc{16.48.1.bq.1} \\ \hline
\mc{16.48.1.br.1} & \mc{16.48.1.bs.1} & \mc{16.48.1.bt.1} & \mc{16.48.1.bu.1} &
\mc{16.48.1.bv.1} & \mc{16.48.1.ca.1} \\ \hline
\mc{16.48.1.cc.1} & \mc{16.48.1.cd.1} & \mc{16.48.1.cf.1} & \mc{16.48.1.cg.1} &
\mc{16.48.1.ch.1} & \mc{16.48.1.cj.1} \\ \hline
\mc{16.48.1.cj.2} & \mc{16.48.1.cl.1} & \mc{16.48.1.cl.2} & \mc{16.48.1.cn.1} &
\mc{16.48.1.co.1} & \mc{16.48.1.cp.1} \\ \hline
\mc{16.48.1.cr.1} & \mc{16.48.1.cs.1} & \mc{16.48.1.ct.1} & \mc{16.48.1.cv.1} &
\mc{16.48.1.cv.2} & \mc{16.48.1.cx.1} \\ \hline
\mc{16.48.1.cx.2} & \mc{16.48.1.d.1} & \mc{16.48.1.db.1} & \mc{16.48.1.dc.1} &
\mc{16.48.1.de.1} & \mc{16.48.1.df.1} \\ \hline
\mc{16.48.1.g.1} & \mc{16.48.1.h.1} & \mc{16.48.1.i.1} & \mc{16.48.1.k.1} &
\mc{16.48.1.p.1} & \mc{16.48.1.q.1} \\ \hline
\mc{16.48.1.q.2} & \mc{16.48.1.r.1} & \mc{16.48.1.r.2} & \mc{16.48.1.t.1} &
\mc{16.48.1.t.2} & \mc{16.48.1.u.1} \\ \hline
\mc{16.48.1.u.2} & \mc{16.48.1.v.1} & \mc{16.48.1.v.2} & \mc{16.48.1.x.1} &
\mc{16.48.1.x.2} & \mc{16.48.1.y.1} \\ \hline
\mc{16.96.1.b.1} & \mc{16.96.1.b.2} & \mc{16.96.1.c.2} & \mc{16.96.1.e.2} &
\mc{16.96.1.f.1} & \mc{16.96.1.f.2} \\ \hline
\mc{16.96.1.g.1} & \mc{16.96.1.l.2} & \mc{16.96.1.m.1} & \mc{16.96.1.m.2} &
\mc{16.96.1.q.1} & \mc{17.18.1.a.1} \\ \hline
\mc{17.36.1.a.1} & \mc{17.36.1.a.2} & \mc{17.72.1.a.1} & \mc{17.72.1.a.2} &
\mc{17.72.1.b.1} & \mc{17.72.1.b.2} \\ \hline
\mc{19.20.1.a.1} & \mc{19.60.1.a.1} & \mc{19.60.1.a.2} & \mc{19.60.1.b.1} &
\mc{27.36.1.a.1} & \mc{27.36.1.b.1} \\ \hline
\mc{27.108.1.a.1} & \mc{27.108.1.a.2} & \mc{32.48.1.a.1} & \mc{32.48.1.a.2} &
\mc{32.48.1.b.1} & \mc{32.48.1.b.2} \\ \hline
\mc{32.96.1.a.1} & \mc{32.96.1.a.2} & \mc{32.96.1.b.1} & \mc{32.96.1.d.1} &
\mc{32.96.1.d.2} & \mc{32.96.1.e.1} \\ \hline
\mc{32.96.1.f.1} & \mc{32.96.1.f.2} & \mc{32.96.1.h.1} &
\mc{49.56.1.a.1} & \mc{27.108.1.b.1} & \mc{7.42.1.a.1}\\ \hline
\mc{8.24.1.bc.1} & \mc{8.24.1.e.1} & \mc{8.24.1.g.1} &
\mc{8.24.1.h.1} & \mc{8.24.1.i.1} & \mc{8.24.1.j.1} \\ \hline
\mc{8.24.1.k.1} & \mc{8.24.1.p.1} & \mc{8.24.1.q.1} & \mc{8.24.1.v.1} & \mc{8.24.1.z.1} & \mc{8.32.1.a.1} \\ \hline
\mc{8.32.1.b.1} & \mc{8.32.1.c.1} & \mc{8.32.1.d.1} & \mc{8.48.1.a.1} & \mc{8.48.1.b.1} & \mc{8.48.1.ba.1} \\ \hline
\mc{8.48.1.bd.1} & \mc{8.48.1.be.1} & \mc{8.48.1.bf.1} & \mc{8.48.1.bg.1} & \mc{8.48.1.bh.1} & \mc{8.48.1.bh.2} \\ \hline
\mc{8.48.1.bj.1} & \mc{8.48.1.bk.1} & \mc{8.48.1.bl.1} & \mc{8.48.1.bm.1} & \mc{8.48.1.bo.1} & \mc{8.48.1.bq.1} \\ \hline \mc{8.48.1.br.1} & \mc{32.96.1.g.2} & \mc{32.96.1.h.2} & \mc{32.96.1.g.1} &  \mc{8.48.1.bt.1} & \mc{8.48.1.bu.1} \\ \hline
\mc{8.48.1.d.1} & \mc{8.48.1.e.1} & \mc{8.48.1.e.2} & \mc{8.48.1.f.1} & \mc{8.48.1.j.1} & \mc{8.48.1.l.1} \\ \hline
\mc{8.48.1.o.1} & \mc{8.48.1.q.1} &  \mc{8.48.1.r.1} & \mc{8.48.1.s.1} & \mc{8.48.1.s.2} & \mc{8.48.1.u.1} \\ \hline
\mc{8.48.1.v.1} & \mc{8.48.1.w.1} & \mc{8.48.1.x.1} & \mc{8.48.1.y.1} & \mc{8.48.1.z.1} & \mc{8.96.1.a.1} \\ \hline
\mc{8.96.1.a.2} & \mc{8.96.1.b.1} & \mc{8.96.1.b.2} & \mc{8.96.1.c.1} & \mc{8.96.1.c.2} & \mc{8.96.1.e.1} \\ \hline
\mc{32.96.1.e.2} & \mc{32.96.1.c.2} & \mc{32.96.1.c.1} & \mc{32.96.1.b.2} & \mc{8.96.1.f.1} & \mc{8.96.1.h.1} \\ \hline
\mc{8.96.1.h.2} & \mc{8.96.1.i.1} & \mc{8.96.1.i.2} & \mc{8.96.1.j.2} & \mc{8.96.1.k.1} & \mc{8.96.1.l.2} \\ \hline
\mc{8.96.1.n.1} & \mc{16.96.1.q.2} & \mc{16.96.1.r.1} & \mc{16.96.1.r.2} & \mc{16.96.1.s.1} & \mc{16.96.1.s.2} \\ \hline
\mc{16.48.1.bb.1} & \mc{16.48.1.bc.1} & \mc{16.48.1.be.1} & \mc{16.48.1.bh.1} & \mc{16.48.1.bi.1} & \mc{16.48.1.bj.1} \\ \hline 
\mc{16.48.1.bk.1} & \mc{16.48.1.bw.1} & \mc{16.48.1.bx.1} & \mc{16.48.1.by.1} & \mc{16.48.1.bz.1} & \mc{16.48.1.c.1} \\ \hline \mc{16.48.1.cb.1} & \mc{16.48.1.ce.1} & \mc{16.48.1.ci.1} & \mc{16.48.1.ci.2} & \mc{16.48.1.ck.1} & \mc{16.48.1.ck.2} \\ \hline 
\mc{16.48.1.cm.1} & \mc{16.48.1.cq.1} & \mc{16.48.1.cu.1} & \mc{16.48.1.cu.2} & \mc{16.48.1.cw.1} & \mc{16.48.1.cw.2} \\ \hline 
\mc{16.48.1.cy.1} & \mc{16.48.1.cz.1} & \mc{16.48.1.da.1} & \mc{16.48.1.dd.1} & \mc{16.48.1.e.1} & \mc{16.48.1.f.1} \\ \hline \mc{16.48.1.j.1} & \mc{16.48.1.m.1} & \mc{16.48.1.n.1} & \mc{16.48.1.o.1} & \mc{16.48.1.s.1} & \mc{16.48.1.s.2} \\ \hline
\mc{16.48.1.w.1} & \mc{16.48.1.w.2} & \mc{16.48.1.z.1} & \mc{16.96.1.a.1} & \mc{16.96.1.a.2} & \mc{16.96.1.c.1} \\ \hline
\mc{16.96.1.d.1} & \mc{16.96.1.d.2} & \mc{16.96.1.e.1} & \mc{16.96.1.g.2} & \mc{16.96.1.h.1} & \mc{16.96.1.h.2} \\ \hline
\mc{16.96.1.i.1} & \mc{16.96.1.i.2} & \mc{16.96.1.j.1} & \mc{16.96.1.j.2} & \mc{16.96.1.l.1} & \mc{16.96.1.n.1} \\ \hline
\mc{16.96.1.n.2} & \mc{16.96.1.o.1} & \mc{16.96.1.o.2} & \mc{16.96.1.p.1} & \mc{16.96.1.p.2} & \mc{16.48.1.l.1}  \\ \hline

\end{longtable}

\begin{longtable}[H]{|p{2.5 cm}|p{2.5 cm}|p{2.5 cm}|p{2.5 cm}|p{2.5 cm}|p{2.5 cm}|}
\caption{LMFDB labels of hyperelliptic prime power level modular curves.\label{tab:hyp}}\\
 \hline
\hline
 \multicolumn{6}{| p{2.5 cm} |}{}\\

\mc{8.48.2.a.1} & \mc{16.48.2.a.1} & \mc{16.48.2.b.1} & \mc{16.48.2.e.1} & \mc{16.48.2.f.1} & \mc{16.48.2.l.1} \\ \hline
\mc{16.48.2.l.2} & \mc{16.48.2.m.1} & \mc{16.48.2.m.2} & \mc{16.48.2.ba.1} & \mc{16.48.2.bb.1} & \mc{16.48.2.bc.1} \\ \hline
\mc{16.48.2.bd.1} & \mc{16.48.2.be.1} & \mc{16.48.2.bf.1} & \mc{16.48.2.bg.1} & \mc{16.48.2.bh.1} & \mc{16.48.2.bi.1} \\ \hline
\mc{16.48.2.bj.1} & \mc{16.48.2.bk.1} & \mc{16.48.2.bl.1} & \mc{16.48.2.bm.1} & \mc{16.48.2.bn.1} & \mc{16.48.2.y.1} \\ \hline
\mc{16.48.2.z.1} & \mc{9.36.2.a.1} & \mc{9.54.2.a.1} & \mc{9.54.2.b.1} & \mc{11.66.2.a.1} & \mc{13.84.2.a.1} \\ \hline
\mc{13.84.2.a.2} & \mc{13.84.2.b.1} & \mc{13.84.2.b.2} & \mc{13.84.2.c.1} & \mc{13.84.2.c.2} & \mc{16.24.2.a.1} \\ \hline
\mc{16.24.2.a.2} & \mc{16.24.2.b.1} & \mc{16.24.2.b.2} & \mc{16.24.2.c.1} & \mc{16.24.2.d.1} & \mc{16.24.2.e.1} \\ \hline
\mc{16.24.2.f.1} & \mc{16.48.2.c.1} & \mc{16.48.2.c.2} & \mc{16.48.2.d.1} & \mc{16.48.2.d.2} & \mc{16.48.2.g.1} \\ \hline
\mc{16.48.2.g.2} & \mc{16.48.2.h.1} & \mc{16.48.2.i.1} & \mc{16.48.2.j.1} & \mc{16.48.2.k.1} & \mc{16.48.2.n.1} \\ \hline
\mc{16.48.2.o.1} & \mc{16.48.2.p.1} & \mc{16.48.2.bo.1} & \mc{16.48.2.bp.1} & \mc{16.48.2.bq.1} & \mc{16.48.2.br.1} \\ \hline
\mc{16.48.2.bs.1} & \mc{16.48.2.bt.1} & \mc{16.48.2.bu.1} & \mc{16.48.2.bv.1} & \mc{16.48.2.q.1} & \mc{16.48.2.r.1} \\ \hline
\mc{16.48.2.s.1} & \mc{16.48.2.t.1} & \mc{16.48.2.u.1} & \mc{16.48.2.v.1} & \mc{16.48.2.w.1} & \mc{16.48.2.x.1} \\ \hline
\mc{16.48.2.bw.1} & \mc{16.48.2.bx.1} & \mc{16.64.2.a.1} & \mc{16.96.2.a.1} & \mc{16.96.2.b.1} & \mc{16.96.2.c.1} \\ \hline
\mc{16.96.2.d.1} & \mc{32.96.2.g.1} & \mc{32.96.2.h.1} & \mc{16.96.2.i.1} & \mc{16.96.2.i.2} & \mc{16.96.2.j.1} \\ \hline
\mc{16.96.2.j.2} & \mc{16.96.2.k.1} & \mc{16.96.2.k.2} & \mc{16.96.2.l.1} & \mc{16.96.2.l.2} & \mc{16.96.2.e.1} \\ \hline
\mc{16.96.2.e.2} & \mc{16.96.2.f.1} & \mc{16.96.2.f.2} & \mc{16.96.2.g.1} & \mc{16.96.2.g.2} & \mc{16.96.2.h.1} \\ \hline
\mc{16.96.2.h.2} & \mc{23.24.2.a.1} & \mc{25.50.2.a.1} & \mc{25.75.2.a.1} & \mc{27.36.2.a.1} & \mc{27.36.2.b.1} \\ \hline
\mc{29.30.2.a.1} & \mc{31.32.2.a.1} & \mc{32.48.2.a.1} & \mc{32.48.2.a.2} & \mc{32.48.2.b.1} & \mc{32.48.2.b.2} \\ \hline
\mc{32.96.2.a.1} & \mc{32.96.2.b.1} & \mc{32.96.2.c.1} & \mc{32.96.2.d.1} & \mc{32.96.2.e.1} & \mc{32.96.2.f.1} \\ \hline
\mc{32.96.2.i.1} & \mc{32.96.2.i.2} & \mc{32.96.2.j.1} & \mc{32.96.2.j.2} & \mc{32.96.2.k.1} & \mc{32.96.2.k.2} \\ \hline 
\mc{32.96.2.l.1} & \mc{32.96.2.l.2} & \mc{37.38.2.a.1} & \mc{8.96.3.f.1} & \mc{8.96.3.f.2} & \mc{8.96.3.h.1} \\ \hline
\mc{8.96.3.h.2} & \mc{8.96.3.i.1} & \mc{8.96.3.j.1} & \mc{16.96.3.b.1} & \mc{16.96.3.bg.1} & \mc{16.96.3.bh.1} \\ \hline
\mc{16.96.3.bj.1} & \mc{16.96.3.bk.1} & \mc{16.96.3.bq.1} & \mc{16.96.3.bs.1} & \mc{16.96.3.bt.1} & \mc{16.96.3.bv.1} \\ \hline
\mc{16.96.3.c.1} & \mc{16.96.3.cm.1} & \mc{16.96.3.cm.2} & \mc{16.96.3.cr.1} & \mc{16.96.3.cr.2} & \mc{16.96.3.cx.1} \\ \hline
\mc{16.96.3.da.1} & \mc{16.96.3.df.1} & \mc{16.96.3.df.2} & \mc{16.96.3.dh.1} & \mc{16.96.3.dh.2} & \mc{16.96.3.di.1} \\ \hline
\mc{16.96.3.dk.1} & \mc{16.96.3.h.1} & \mc{16.96.3.i.1} & \mc{16.48.3.a.1} & \mc{16.48.3.a.2} & \mc{16.48.3.bc.1} \\ \hline
\mc{16.48.3.bc.2} & \mc{16.48.3.bf.1}& \mc{16.48.3.bf.2}& \mc{16.48.3.e.1}& \mc{16.48.3.e.2}& \mc{16.48.3.i.1} \\ \hline
\mc{16.48.3.i.2} & \mc{16.48.3.m.1}&\mc{16.48.3.m.2}&\mc{16.48.3.t.1} & \mc{16.48.3.t.2}& \mc{16.48.3.w.1} \\ \hline
\mc{16.48.3.w.2} & \mc{16.48.3.bm.1}&\mc{16.48.3.bn.1}&\mc{16.48.3.bo.1} & \mc{16.48.3.bq.1}& \mc{16.48.3.br.1} \\ \hline
\mc{16.48.3.bs.1}& \mc{16.48.3.bx.1}&\mc{16.48.3.bz.1}&\mc{16.48.3.ca.1}&\mc{16.48.3.cb.1}&\mc{16.48.3.cd.1} \\ \hline
\mc{16.48.3.ce.1}& \mc{16.48.3.ba.1}& \mc{16.48.3.bb.1}& \mc{16.48.3.c.1}& \mc{16.48.3.c.2} & \mc{16.48.3.f.1} \\ \hline
\mc{16.48.3.f.2} & \mc{16.48.3.h.1}&\mc{16.48.3.k.1}&\mc{16.48.3.o.1}&\mc{16.48.3.q.1}&\mc{16.48.3.r.1}\\ \hline
\mc{16.48.3.u.1} & \mc{16.48.3.y.1}&\mc{16.48.3.cf.1}&\mc{16.48.3.cg.1}&\mc{16.48.3.bi.1}&\mc{16.48.3.bi.2} \\ \hline
\mc{16.48.3.bk.1} & \mc{16.48.3.bk.2}&\mc{16.48.3.bu.1}& \mc{16.48.3.bu.2}& \mc{16.48.3.bw.1}&\mc{16.48.3.bw.2}\\ \hline
\mc{16.96.3.a.1}& \mc{16.96.3.a.2}&\mc{16.96.3.bm.1}&\mc{16.96.3.bn.1}&\mc{16.96.3.bw.1}&\mc{16.96.3.bx.1} \\ \hline
\mc{16.96.3.ca.1}& \mc{16.96.3.cb.1}&\mc{16.96.3.ce.1}&\mc{16.96.3.cf.1}&\mc{16.96.3.ci.1}&\mc{16.96.3.cj.1} \\ \hline
\mc{16.96.3.ck.1} & \mc{16.96.3.cl.1}& \mc{16.96.3.dp.1}& \mc{16.96.3.dq.1}&\mc{16.96.3.dt.1}&\mc{16.96.3.du.1} \\ \hline
\mc{16.96.3.dx.1} & \mc{16.96.3.dy.1}& \mc{16.96.3.e.1} & \mc{16.96.3.e.2}&\mc{16.96.3.p.1} & \mc{16.96.3.p.2} \\ \hline
\mc{16.96.3.r.1} & \mc{16.96.3.r.2} & \mc{16.96.3.t.1} & \mc{16.96.3.t.2}&\mc{16.96.3.y.1} & \mc{16.96.3.y.2} \\ \hline
\mc{16.96.3.ba.1}& \mc{16.96.3.ba.2} & \mc{16.96.3.cc.1}&\mc{16.96.3.cg.1} & \mc{16.96.3.cp.1}&\mc{16.96.3.cp.2} \\ \hline
\mc{16.96.3.cs.1} & \mc{16.96.3.cs.2} & \mc{16.96.3.dr.1}&\mc{16.96.3.dv.1} & \mc{16.96.3.v.1}&\mc{16.96.3.v.2} \\ \hline
\mc{16.96.3.w.1} & \mc{16.96.3.w.2} & \mc{16.96.3.z.1} & \mc{16.96.3.z.2} & \mc{16.96.3.ey.1} & \mc{16.96.3.ey.2} \\ \hline
\mc{16.96.3.fa.1} & \mc{16.96.3.fa.2} & \mc{32.48.3.e.1} & \mc{32.48.3.e.2} & \mc{32.48.3.f.1} & \mc{32.48.3.f.2} \\ \hline
\mc{32.48.3.c.1} & \mc{32.48.3.c.2} & \mc{32.48.3.d.1} & \mc{32.48.3.d.2} & \mc{32.96.3.b.1} & \mc{32.96.3.b.2} \\ \hline
\mc{32.96.3.c.1} & \mc{32.96.3.c.2} & \mc{32.96.3.f.1} & \mc{32.96.3.f.2} & \mc{32.96.3.h.1} & \mc{32.96.3.h.2} \\ \hline
\mc{32.96.3.i.1} & \mc{32.96.3.i.2} & \mc{32.96.3.l.1} & \mc{32.96.3.l.2} & \mc{32.96.3.n.1} & \mc{32.96.3.n.2} \\ \hline
\mc{32.96.3.q.1} & \mc{32.96.3.q.2} & \mc{32.96.3.r.1} & \mc{32.96.3.r.2} & \mc{32.96.3.s.1} & \mc{32.96.3.s.2} \\ \hline
\mc{32.96.3.u.1} & \mc{32.96.3.u.2} & \mc{32.96.3.v.1} & \mc{32.96.3.v.2} & \mc{32.96.3.w.1} & \mc{32.96.3.w.2} \\ \hline
\mc{32.96.3.y.1} & \mc{32.96.3.y.2} & \mc{41.42.3.a.1} & \mc{64.96.3.c.1}& \mc{64.96.3.c.2} & \mc{64.96.3.d.1} \\ \hline \mc{64.96.3.d.2} & \mc{47.48.4.a.1} & \mc{59.60.5.a.1} & \mc{71.72.6.a.1} & \mc{32.96.7.a.1} & \mc{32.96.7.g.1} \\ \hline
\mc{32.96.7.r.1} & \mc{32.96.7.s.1} & \mc{64.96.7.g.1} & \mc{64.96.7.g.2} & \mc{64.96.7.h.1} & \mc{64.96.7.h.2} \\ \hline
\end{longtable}

\begin{longtable}[H]{|p{2.5cm}|p{2.5cm}|p{2.5cm}|p{2.5cm}|}
\caption{Modular curves which map to a positive rank elliptic curve in LMFDB database.\label{tab:LMFDB deductions}}\\
 \hline
 
LMFDB label of curve & \multicolumn{1}{l|}{\begin{tabular}[c]{@{}l@{}}LMFDB label of positive rank\\ elliptic curve that it maps to\end{tabular}} & LMFDB label of curve & \multicolumn{1}{l|}{\begin{tabular}[c]{@{}l@{}}LMFDB label of positive rank\\ elliptic curve that it maps to\end{tabular}} \\
 \hline
 \endfirsthead

 \hline
 \multicolumn{4}{|c|}{Continuation of Table \ref{tab:LMFDB deductions}}\\
 \hline
 LMFDB label of curve & \multicolumn{1}{l|}{\begin{tabular}[c]{@{}l@{}}LMFDB label of positive rank\\ elliptic curve that it maps to\end{tabular}} & LMFDB label of curve & \multicolumn{1}{l|}{\begin{tabular}[c]{@{}l@{}}LMFDB label of positive rank\\ elliptic curve that it maps to\end{tabular}} \\
 \hline
 \endhead

 \hline
 \endfoot

 \hline
 \multicolumn{4}{| p{2.5cm} |}{}\\

\mc{16.48.3.b.1}   &  \mc{16.24.1.n.1}  &   \mc{16.48.3.bd.2}     &   \mc{16.24.1.n.1}           \\  \hline
   \mc{16.48.3.b.2}  &   \mc{16.24.1.n.2}  &\mc{16.96.5.bd.2} &  \mc{16.48.1.cx.2}        \\  \hline
 \mc{16.48.3.bd.1}    &  \mc{16.24.1.n.2}    &\mc{16.96.5.bd.3} & \mc{16.48.1.cx.1}          \\  \hline

\mc{16.48.3.bg.1}     &   \mc{16.24.1.n.1}  &\mc{16.96.5.ca.1} & \mc{16.48.1.cx.1}         \\  \hline
\mc{16.48.3.n.1}     &    \mc{16.24.1.n.1} &\mc{16.96.5.ca.2} &\mc{16.48.1.cx.2}           \\  \hline
\mc{16.96.3.ep.1} & \mc{16.48.1.bl.1} & \mc{16.96.5.x.2} & \mc{16.48.1.ct.1} \\  \hline
\mc{16.96.3.es.1} & \mc{16.48.1.bn.1} & \mc{16.96.5.i.1} & \mc{16.48.1.df.1} \\  \hline
\mc{16.96.3.et.1} & \mc{16.48.1.cs.1} &\mc{16.96.5.ce.1} & \mc{16.48.1.cx.2}\\ \hline
\mc{16.96.3.eu.1} & \mc{16.48.1.bn.1} & \mc{16.96.5.ce.2}& \mc{16.48.1.cx.1}\\ \hline
\mc{16.96.3.ev.1} & \mc{16.48.1.bl.1}  &  \mc{16.96.5.i.2} & \mc{16.48.1.cr.1} \\ \hline
\mc{16.96.3.ew.1} & \mc{16.48.1.cs.1}& \mc{16.96.5.d.2} & \mc{16.48.1.cs.1} \\ \hline
\mc{16.96.5.d.1} & \mc{16.48.1.dc.1}  &\mc{16.96.5.w.1} & \mc{16.48.1.cx.1}\\ \hline
\mc{16.96.3.fd.1} & \mc{16.48.1.bv.1}& \mc{16.96.5.w.2} & \mc{16.48.1.cx.2}\\ \hline
\mc{16.96.3.ff.1} & \mc{16.48.1.bv.1}&\mc{16.96.5.da.1} & \mc{16.48.1.bn.1}\\ \hline
\mc{16.96.3.fg.1} & \mc{16.48.1.dc.1} &\mc{16.96.5.db.1} & \mc{16.48.1.bl.1}\\ \hline
\mc{16.96.3.fh.1} & \mc{16.48.1.dc.1} & \mc{16.96.5.dd.1}& \mc{16.48.1.bl.1}\\ \hline
\mc{16.96.3.fi.1} & \mc{16.48.1.dc.1}& \mc{16.96.5.dg.1} & \mc{16.48.1.bn.1}\\ \hline
\mc{16.96.3.eh.1} & \mc{16.48.1.de.1}& \mc{16.96.5.dh.1}& \mc{16.48.1.bv.1} \\ \hline
\mc{16.96.3.ei.1} & \mc{16.48.1.df.1}&  \mc{16.96.5.cq.1} & \mc{16.48.1.dc.1}  \\ \hline
\mc{16.96.3.ej.1} & \mc{16.48.1.de.1}& \mc{16.96.5.cp.1} & \mc{16.48.1.de.1}  \\ \hline
\mc{16.96.3.ek.1} & \mc{16.48.1.df.1}&\mc{16.96.5.dn.1} & \mc{16.48.1.bv.1}\\ \hline
 
\mc{16.96.3.fk.1} & \mc{16.48.1.cf.1}& \mc{16.96.5.dq.1}& \mc{16.48.1.cf.1}\\ \hline
\mc{16.96.3.fl.1} & \mc{16.48.1.cg.1}& \mc{16.96.5.dr.1} & \mc{16.48.1.cg.1}\\ \hline
\mc{16.96.3.fm.1} & \mc{16.48.1.ch.1}& \mc{16.96.5.ds.1}& \mc{16.48.1.ch.1}\\ \hline
\mc{16.96.3.fn.1} & \mc{16.48.1.de.1}& \mc{16.96.4.k.1} & \mc{16.48.1.ca.1} \\ \hline
\mc{16.96.3.fo.1} & \mc{16.48.1.cf.1}& \mc{16.96.5.du.1}&\mc{16.48.1.cf.1}\\ \hline
\mc{16.96.3.fp.1} & \mc{16.48.1.cg.1}& \mc{16.96.5.dv.1}& \mc{16.48.1.cg.1}\\ \hline
\mc{16.96.3.fq.1} & \mc{16.48.1.ch.1}& \mc{16.96.5.dw.1}& \mc{16.48.1.ch.1}\\ \hline
\mc{32.96.3.bh.1} & \mc{16.48.1.bl.1}&\mc{32.96.5.bm.1} & \mc{16.48.1.bq.1}\\ \hline
\mc{32.96.3.bi.1} & \mc{16.48.1.bl.1}& \mc{32.96.5.bn.1}& \mc{16.48.1.bq.1}\\ \hline
\mc{32.96.3.bj.1} & \mc{16.48.1.bq.1}& \mc{32.96.5.bo.1}& \mc{16.48.1.bv.1}\\ \hline
\mc{32.96.3.bk.1} & \mc{16.48.1.bq.1}& \mc{32.96.5.bp.1}& \mc{16.48.1.bv.1}\\ \hline
\mc{32.96.3.bl.1} & \mc{16.48.1.bv.1}&\mc{32.96.5.bk.1} & \mc{16.48.1.bl.1} \\ \hline
\mc{32.96.3.bm.1} & \mc{16.48.1.bv.1}&\mc{32.96.5.bl.1} & \mc{16.48.1.bl.1}\\ \hline
\mc{32.96.3.bx.1} & \mc{16.48.1.cs.1}& \mc{32.96.5.bt.1} & \mc{16.48.1.cs.1}\\ \hline
\mc{32.96.3.ca.1} & \mc{16.48.1.dc.1}& \mc{32.96.5.bu.1}& \mc{16.48.1.cs.1}\\ \hline
\mc{32.96.3.cb.1} & \mc{16.48.1.de.1}& \mc{32.96.5.bv.1}& \mc{16.48.1.dc.1}\\ \hline
\mc{32.96.3.cc.1} & \mc{16.48.1.df.1}& \mc{32.96.5.by.1} & \mc{16.48.1.dc.1}\\ \hline
\mc{32.96.3.bn.1} & \mc{16.48.1.cc.1}&\mc{32.96.5.bw.1} & \mc{16.48.1.de.1}\\ \hline
\mc{32.96.3.bo.1} & \mc{16.48.1.cc.1}& \mc{32.96.5.bx.1}& \mc{16.48.1.df.1}\\ \hline
\mc{32.96.3.bq.1} & \mc{16.48.1.ch.1}& \mc{32.96.5.bz.1}& \mc{16.48.1.de.1}\\ \hline
\mc{32.96.3.bs.1} & \mc{16.48.1.ch.1}& \mc{32.96.5.ca.1}& \mc{16.48.1.df.1}\\ \hline
\mc{32.96.3.bp.1} & \mc{16.48.1.cg.1}& \mc{32.96.5.bj.1}& \mc{16.48.1.ch.1}\\ \hline
\mc{32.96.3.br.1} & \mc{16.48.1.cg.1}& \mc{32.96.5.bs.1}& \mc{16.48.1.ch.1}\\ \hline
\mc{11.110.4.a.1} & \mc{11.55.1.b.1}&\mc{32.96.5.bh.1} & \mc{16.48.1.cc.1}\\ \hline
\mc{11.110.4.b.1} & \mc{11.55.1.b.1}& \mc{32.96.5.bi.1} & \mc{16.48.1.cg.1}\\ \hline
\mc{16.96.4.b.1} & \mc{16.48.1.bq.1}&\mc{32.96.5.bq.1} & \mc{16.48.1.cc.1}\\ \hline
\mc{16.96.4.c.1} & \mc{16.48.1.bs.1} & \mc{32.96.5.br.1}& \mc{16.48.1.cg.1} \\ \hline
\mc{16.96.4.e.1} & \mc{16.48.1.bq.1} & \mc{32.96.7.be.1}& \mc{16.48.1.cs.1}\\ \hline
\mc{16.96.4.h.1} & \mc{16.48.1.bs.1} & \mc{32.96.7.bh.1}& \mc{16.48.1.dc.1}\\ \hline
\mc{16.96.4.i.1} & \mc{16.48.1.ca.1} & \mc{32.96.7.bi.1}& \mc{16.48.1.de.1}\\ \hline
\mc{16.96.4.j.1} & \mc{16.48.1.ca.1} &\mc{32.96.7.bj.1} & \mc{16.48.1.df.1}\\ \hline
\mc{16.96.4.l.1} & \mc{16.48.1.ca.1} & \mc{16.96.3.dc.1}& \mc{16.48.1.bn.1}\\ \hline
\mc{16.96.4.m.1} & \mc{16.48.1.cc.1} & \mc{16.96.3.dj.1} & \mc{16.48.1.bs.1}\\ \hline
\mc{16.96.4.n.1} & \mc{16.48.1.cc.1} & \mc{16.48.3.cc.1}& \mc{16.24.1.n.2} \\ \hline
\mc{16.96.4.o.1} & \mc{16.48.1.cc.1} & \mc{16.96.3.f.1}& \mc{16.48.1.bs.1}\\ \hline
\mc{16.96.4.p.1} & \mc{16.48.1.cc.1} & \mc{16.96.3.j.1}& \mc{16.48.1.bn.1}\\ \hline
\mc{16.96.5.a.1} & \mc{16.48.1.de.1} & \mc{16.48.3.bp.1}& \mc{16.24.1.n.2}\\ \hline
 
\mc{16.96.5.bg.1} & \mc{16.48.1.dc.1} & \mc{16.48.3.j.1}& \mc{16.24.1.l.1}\\ \hline
\mc{16.96.5.bh.1} & \mc{16.48.1.cr.1} & \mc{16.48.3.z.1}& \mc{16.24.1.l.1} \\ \hline
\mc{16.96.5.bi.1} & \mc{16.48.1.de.1} & \mc{16.48.3.bv.1}& \mc{16.24.1.l.1}\\ \hline
\mc{16.96.5.bj.1} & \mc{16.48.1.ct.1} &\mc{16.48.3.bv.2} & \mc{16.24.1.l.1}\\ \hline
\mc{16.96.5.bs.1} & \mc{16.48.1.cr.1} &\mc{16.96.3.be.1} &\mc{16.48.1.ca.1} \\ \hline

\mc{16.96.5.bt.1} & \mc{16.48.1.cs.1} & \mc{16.96.3.ex.1}& \mc{16.48.1.cx.1} \\ \hline
\mc{16.96.5.bt.2} & \mc{16.48.1.ct.1} & \mc{16.96.3.bp.1}& \mc{16.48.1.cf.1}\\ \hline
\mc{16.96.5.bx.1} & \mc{16.48.1.ct.1} & \mc{16.96.3.dn.1}& \mc{16.48.1.cf.1}\\ \hline
\mc{16.96.5.by.1} & \mc{16.48.1.cr.1} & \mc{16.96.3.do.1} & \mc{16.48.1.ca.1}\\ \hline
 
\mc{16.96.5.ch.2} & \mc{16.48.1.dc.1} & \mc{16.96.3.ex.2} & \mc{16.48.1.cx.2}\\ \hline
\mc{16.96.5.ci.1} & \mc{16.48.1.df.1} & \mc{16.96.3.ez.1}& \mc{16.48.1.cx.2}\\ \hline
\mc{16.96.5.ci.2} & \mc{16.48.1.de.1} & \mc{16.96.3.ez.2}&\mc{16.48.1.cx.1} \\ \hline

\end{longtable}

\begin{longtable}[H]{|p{2.5 cm}|p{3 cm}|p{2.5 cm}|p{2.5 cm}|p{2.5 cm}|p{2.5 cm}|}
\caption{LMFDB labels of positive rank bielliptic prime power level modular curves that are not hyperelliptic and not in Table \ref{tab:LMFDB deductions}.\label{tab:posrankbi}}\\
 \hline
\hline
 \multicolumn{6}{| p{2.5 cm} |}{}\\

\mc{16.48.3.x.1} & \mc{16.96.3.ds.1} & \mc{16.96.3.dw.1} & \mc{16.96.5.bq.1} & \mc{16.96.5.cf.1} & \mc{16.96.5.ej.1} \\ \hline
\mc{16.96.5.l.2} & \mc{16.96.5.s.1} & \mc{27.108.7.a.1} & \mc{27.108.7.g.1} & \mc{27.36.3.a.1} & \mc{32.48.3.b.1} \\ \hline
\mc{32.48.3.b.2} & \mc{32.96.5.a.1} & \mc{32.96.5.a.2} & \mc{32.96.5.be.1} & \mc{32.96.5.be.2} & \mc{32.96.5.c.1} \\ \hline
\mc{32.96.5.f.2} & \mc{32.96.5.h.1} & \mc{32.96.5.i.1} & \mc{32.96.5.m.1} & \mc{32.96.5.n.2} & \mc{32.96.5.p.1} \\ \hline
\mc{32.96.5.p.2} & \mc{37.114.4.a.1} & \mc{64.96.5.a.1} & \mc{64.96.5.a.2} & \mc{64.96.5.d.1} & \mc{81.108.7.a.1} \\ \hline
\mc{43.44.3.a.1} & \mc{53.54.4.a.1} & \mc{61.62.4.a.1} & \mc{79.80.6.a.1} & \mc{83.84.7.a.1} & \mc{89.90.7.a.1} \\ \hline
\mc{101.102.8.a.1} & \mc{131.132.11.a.1} & & & & \\ \hline
 \end{longtable}
\bibliographystyle{amsplain}
\bibliography{references}
\end{document}